\documentclass[11pt]{amsart}

\usepackage{graphicx}
\usepackage{mathrsfs}
\usepackage{color}
\usepackage{xcolor}
\colorlet{red}{black}
\usepackage{amssymb}
\usepackage{amsmath}
\usepackage{mathrsfs}
\usepackage[margin = 1.4in] {geometry}
\usepackage[hyperindex,breaklinks]{hyperref}
\usepackage{float}  

\newtheorem{prop}{Proposition}
\newtheorem{theo}[prop]{Theorem}
\newtheorem*{theo*}{Theorem}
\newtheorem{lemm}[prop]{Lemma}

\newtheorem{rema}[prop]{Remark}
\newtheorem*{Ackn}{Acknowledgements}

\theoremstyle{definition}
\newtheorem{defi}[prop]{Definition}

\makeatletter

\newcommand{\Rmnum}[1]{\expandafter\@slowromancap\romannumeral #1@}
\makeatother

\newcommand{\cA}{\mathcal A}

\newcommand{\cC}{\mathcal C}

\newcommand{\cF}{\mathcal F}

\newcommand{\cH}{\mathcal H}

\newcommand{\cO}{\mathcal O}

\newcommand{\bB}{\mathbb{B}}
\newcommand{\bR}{\mathbb{R}}
\newcommand{\bS}{\mathbb{S}}

\DeclareMathOperator{\Lip}{Lip}
\DeclareMathOperator{\diam}{diam}

\DeclareMathOperator{\Vol}{Vol}

\DeclareMathOperator{\Area}{Area}

\DeclareMathOperator{\Ric}{Ric}

\DeclareMathOperator{\Div}{div}

\DeclareMathOperator{\secondfund}{II}

\newcommand{\bangle}[1]{\left\langle #1 \right\rangle}

\title[volume comparison with boundary]{A spectral volume comparison for manifolds with weakly convex boundary}

\author{Jia Li}
\address{School of Mathematical Sciences\\
		Xiamen University\\
		361005, Xiamen, P.R. China}
\email{lijiamath@stu.xmu.edu.cn}
\thanks{2020 Mathematics Subject Classification: 53C21, 53C24, 58J60.}
\thanks{Key words and phrases: Diameter estimate, Volume comparison, Free boundary $\mu$-bubble, Isoperimetric profile.}

\begin{document}
\begin{abstract}
We establish the Bonnet-Myers theorem and the Bishop-Gromov volume comparison theorem in the spectral sense for manifolds with weakly convex boundary. For $n\geq 3$, let $(M^n,g)$ be a {\color{red}connected} compact smooth  $n$-manifold with weakly convex boundary $\partial M$. If there exists a positive function $w\in C^{\infty}(M)$ that satisfies:
\begin{equation*}
  \begin{cases}
    -\frac{n-1}{n-2}\Delta w+\Lambda_{\Ric} w\geq (n-1)w, \enspace in \enspace M,\\
    \frac{\partial w}{\partial \eta}=0, \enspace\enspace\enspace\enspace \enspace\enspace\enspace\enspace\enspace\enspace\enspace\enspace\enspace\enspace\enspace\enspace  \enspace\enspace \enspace\enspace on \enspace\partial M,
  \end{cases}
\end{equation*}
where $\Lambda_{\Ric}$ denotes the smallest eigenvalue of the Ricci tensor, $\eta$ is the unit co-normal vector field of $\partial M$ in $M$, then the diameter of $M$ satisfies $\diam(M)\leq (\frac{\max w}{\min w})^{\frac{n-3}{n-1}}\pi$.\par If, in addition, $w$ attains its minimum on the boundary $\partial M$, we obtain a sharp upper bound for the volume of $M$: $\Vol(M)\leq \Vol(\bS^n_{+})$, with equality holding if and only if $M^n$ is isometric to the unit round hemisphere $\bS^{n}_{+}$.
\end{abstract}
\maketitle
\section{Introduction}
A classical Bonnet-Myers theorem states that if a complete $n$-dimensional Riemannian manifold $M^n$ has Ricci curvature at least $(n-1)c$, then the diameter of $M$ is at most $\frac{\pi}{\sqrt{\color{red}c}}$, where $c\in\bR$ is a positive constant. Moreover,  for a fixed $p\in M$, let $B_{p}(r)$ be the geodesic ball of radius $r$ centered at $p$ in $M$, and denote $B^{c}(r)$ by the geodesic ball of radius $r$ centered at the origin in the space form of constant sectional curvature $c$. It follows from the Bishop-Gromov volume comparison theorem that the ratio $\frac{\Vol(B_{p}(r))}{\Vol(B^{c}(r))}$ is non-increasing  for any $r\in(0,\infty)$, and the equality holds if and only if $M$ is isometric to $\bS^{n}(\frac{1}{\sqrt{c}})$. Cheng \cite{Cheng-maxdiam} proved that if the diameter of $M$, denoted by $\diam(M)$, is equal to $\frac{\pi}{\sqrt{c}}$, then $M$ must be isometric to the $n$-sphere of constant sectional curvature $c$.

Recently,~Antonelli-Xu \cite{Antonelli-Xu} established a sharp and rigid spectral generalization of both the Bonnet-Myers theorem and the Bishop-Gromov volume comparison theorem, which are precisely stated as follows.
\begin{theo}[\cite{Antonelli-Xu}]\label{theo-1}
 Let $M^n$, $n\geq 3$, be an $n$-dimensional compact smooth manifold, and let $0\leq \theta\leq \frac{n-1}{n-2}$, $\lambda>0$. Let $\Lambda_{\Ric}$ 
 be the smallest eigenvalue of the Ricci tensor. Assume that there is a positive function $w\in C^{\infty}(M)$ such that:
  $$\theta\Delta_{M}w\leq \Lambda_{\Ric} w-(n-1)\lambda w,$$
Then we have:
\begin{itemize}
  \item A diameter bound $$\diam(M)\leq\frac{\pi}{\sqrt{\lambda}}(\frac{\max{w}}{\min{w}})^{\frac{n-3}{n-1}\theta}.$$
  \item A sharp volume bound $$\Vol({M})\leq\lambda^{-\frac{n}{2}}\Vol(\bS^{n}).$$ Moreover, if equality holds, then every function $w$ is constant, and ${M}$ is isometric to the round sphere of radius $\lambda^{-\frac{1}{2}}$.
\end{itemize}
\end{theo}
\begin{rema}
When $\theta=0$, Theorem 1 reduces to the Bonnet-Myers theorem and the Bishop-Gromov volume comparison theorem.
When $n=3$, the above theorem has been established in \cite[Theorem 5.1]{Chodosh-Li-Minter-Stryker-5bernstein}, which can be used to solve the stable Bernstein problem in $\bR^5$.
\end{rema}

Corresponding to the volume comparison and rigidity results for manifolds without boundary, Hang-Wang \cite{hang-wang2009} considered compact manifolds with boundary and positive Ricci curvature, established a boundary version of rigidity results as follows:
\begin{theo}[\cite{hang-wang2009}]
  Let $(M^n,g)$, $n\geq 2$, be a compact $n$-dimensional smooth manifold with non-empty boundary $\partial M$. Suppose that
   $\Lambda_{\Ric}\geq n-1$,
     $(\partial M,g|_{\partial M})$ is isometric to the standard sphere $\bS^{n-1}\subset\bR^{n}$, and $\partial M$ is weakly convex in the sense that its second fundamental form A is non-negative.
  Then $(M^n,g)$ is isometric to the hemisphere $\bS^{n}_{+}$.
\end{theo}

 It is natural to investigate whether Theorem \ref{theo-1} can be extended to compact manifolds with boundary. This is the main aim of this paper.

\begin{theo}\label{main-theo}
   Let $M^n$, $n\geq 3$, be a compact {\color{red}connected}  smooth manifold with weakly convex boundary $\partial M$, and let $0\leq \theta\leq \frac{n-1}{n-2}$, $\lambda>0$. If there exists a positive function $w\in C^{\infty}(M)$ that satisfies:
\begin{equation*}
  \begin{cases}
    \theta\Delta_{M}w\leq \Lambda_{\Ric}w-(n-1)\lambda w, \enspace in \enspace M,\\
   \enspace\enspace\enspace \frac{\partial w}{\partial \eta}=0, \enspace\enspace\enspace\enspace\enspace\enspace\enspace\enspace\enspace\enspace\enspace\enspace\enspace  \enspace\enspace on \enspace\partial M,
  \end{cases}
\end{equation*}
where $\Lambda_{\Ric}$ denotes the smallest eigenvalue of the Ricci tensor, $\eta$ is the unit co-normal vector field of $\partial M$ in $M$,
then we have the diameter estimate: $$\diam(M)\leq\frac{\pi}{\sqrt{\lambda}}(\frac{\max{w}}{\min{w}})^{\frac{n-3}{n-1}\theta}.$$If, in addition, $w$ attains its minimal value on $\partial M$, then the volume of $M$ satisfies:
\begin{align}\label{vol-1}
\Vol(M)\leq\lambda^{-\frac{n}{2}}\Vol(\bS^{n}_{+}).
\end{align}
Moreover, if inequality (\ref{vol-1}) becomes an equality, then every function $w$ is constant, and $M$ is isometric to the round hemisphere of radius $\lambda^{-\frac{1}{2}}$.
\end{theo}
For simplicity, we let $\theta=\frac{n-1}{n-2}$, $\lambda=1$. We say that the Ricci curvature of $M$ has a positive lower bound in the spectral sense if
$$\lambda_{1}(-\frac{n-1}{n-2}\Delta_{M}+\Lambda_{\Ric})\geq (n-1),$$
where $\lambda_{1}$ denotes the first eigenvalue that satisfies the Neumann boundary condition. Hence, for any compactly supported function $\varphi$, the inequality $$\int_{M}\frac{n-1}{n-2}|\nabla \varphi|^2+(\Lambda_{\Ric}-(n-1))\varphi^{2}\geq 0$$ holds on $M$. Then, 
there exists a positive function $w\in C^{\infty}(M)$ satisfying   \begin{equation*}
  \begin{cases}
    -\frac{n-1}{n-2}\Delta_{M} w+\Lambda_{\Ric} w\geq (n-1)w, \enspace in \enspace M,\\
    \frac{\partial w}{\partial \eta}=0, \enspace\enspace\enspace\enspace \enspace\enspace\enspace\enspace\enspace\enspace\enspace\enspace\enspace\enspace\enspace\enspace  \enspace\enspace \enspace\enspace on \enspace\partial M.
  \end{cases}
\end{equation*}
As a consequence, we have the following Theorem \ref{theo-5} in terms of the first eigenvalue satisfying Neumann boundary condition.
\begin{theo}\label{theo-5}
    Let $M^n$, $n\geq 3$, be a compact {\color{red}connected}  smooth manifold with weakly convex boundary $\partial M$, if
    $$\lambda_{1}(-\frac{n-1}{n-2}\Delta_{M}+\Lambda_{\Ric})\geq (n-1),$$
    and the corresponding eigenfunction achieves its minimum in $\partial M$, then
 $\Vol(M)\leq\Vol(\bS^{n}_{+})$. When equality holds, $M$ is isometric to the round hemisphere $\bS^{n}_{+}$.
\end{theo}

Let us first sketch the proof of Theorem 1. $\mu$-bubble and isoperimetric  profile are two key ingredients in the proof of Theorem 1. Firstly, Antonelli-Xu \cite{Antonelli-Xu} constructed an unequally weighted $\mu$-bubble, by the non-negativity of the second variation of the functional $\cA$ (see Section 2.1), they concluded directly the diameter estimate of $M$. The $\mu$-bubble, which was initially introduced by Gromov in \cite{Gromov2019Fourlecture}, has been successfully used to {\color{red}solve} stable Bernstein problems and geometric rigidity problems such as \cite{Chodosh-Li-anisotropic,Chdosh-Li-Stryker-positive cueved,Chodosh-Li-Minter-Stryker-5bernstein,Hong-cmc nonexis,Mazet-6Bernstein}. They constructed the unequally isoperimetric profile, established a differential inequality according to the non-negativity of the second order variation, by ODE comparison, and finally obtained an upper bound of $\Vol(M)$. The isoperimetric profile originally due to Bray \cite{Bray97}, was first used by Chodosh-Li-Minter-Stryker in \cite{Chodosh-Li-Minter-Stryker-5bernstein} to establish a spectral volume comparison theorem for $3$-dimensional closed manifold. In fact, Antonelli-Xu extended the spectral volume comparison of \cite{Chodosh-Li-Minter-Stryker-5bernstein} to any dimension.

As we have seen, the content of Theorem 1 is closely related to the stable Bernstein problem and geometric rigidity problems. It also provides an effective tool for addressing Bernstein problems. Soon, combining the strategy of \cite{Chodosh-Li-Minter-Stryker-5bernstein} with Theorem \ref{theo-1}, Mazet \cite{Mazet-6Bernstein} successfully solved the stable Bernstein problem in $\bR^6$. In addition, Theorem 1 has also been applied to other geometric rigidity problems such as \cite{Hong-cmc nonexis,L-Xia}.

The proof of main Theorem \ref{main-theo} basically follows the Antonelli-Xu's approach, because $M$ has compact boundary, there will be extra terms in the integral involving boundary. To solve the problem, we require that the boundary $\partial M$ is weakly convex, i.e., the second fundamental form of $\partial M$ with respect to $\eta$ is non-negative. In addition, we also require that $w$ attains its minimal value on the boundary $\partial M$, such that the upper bound of $M$ is independent of $w$. We also note that the $\mu$-bubble and isoperimetric profile may exhibit singularities when $n\geq 8$. On the other hand, singularities may occur in the interior or the boundary of isometric profile. First, for the regular case $n\leq 7$, we basically follow the method of Theorem 1, except for some extra boundary integrals. For $n\geq 8$, the occurrence of singularities makes the free boundary minimizer and isoperimetric profile worse. We adopt the method of Antonelli-Xu \cite{Antonelli-Xu} and Bray \cite{Bray-Gui-Liu-Zhang} and make an extension to the free boundary case. We aim to estimate the size of small neighborhoods around the singular sets, which may occur in the interior or the boundary of the isoperimetric profile. The only difficulty is to deal with the singular sets around boundary points. Therefore, we need to make some changes based on known methods.

To achieve it, we extend the monotonicity formula about stationary free boundary varifolds, which has been proved by Guang-Li-Zhou in \cite{Guang-Li-Zhou2018}, to the free boundary varifolds with generalized mean curvature bounded above. Once we estimate the size of the neighborhoods of singular sets, we will construct a geometric flow fixing the singular sets on the isoperimetric profile (or $\mu$-bubble). Finally, using the partition of unity theorem, we construct a cut-off function such that it vanishes in the singular sets and equals $1$ outside the small neighborhoods of the singular sets. Multiplying this cutoff function with the outward unit normal vector field, the required flow can be obtained.

The paper is organized as follows. In Section 2, we give some preliminary and auxiliary results that will be used in this paper. In Section 3, we construct a free boundary $\mu$-bubble and derive the diameter estimate of $M$. Then, we construct the unequally weighted free boundary isoperimetric profile, obtain the volume upper bound of $M$. In Section 4, we focus on the singular cases and extend the Antonelli-Xu's method to the case of isoperimetric profile with free boundary.
 \begin{Ackn}
The author thanks his advisor, Prof. Chao Xia, for continuous guidance and support; Prof. Han Hong for pointing out writing errors and for helpful discussions; and the referees for their constructive comments and suggestions.
\end{Ackn}

\section{PRELIMINARIES}
{\color{red}First, we recall the notions of Caccioppoli sets and free boundary $\mu$-bubbles in a Riemannian manifold with non-empty boundary. Some basic consequences are given.
\begin{defi}
    A measurable set $E$ in a compact Riemannian manifold $N^{n}$ is called a Caccioppoli set (or a set of finite perimeter) if its characteristic function $\chi_{E}$ is a function of bounded variation, i.e., the following is finite:
    \[
    P(E):=\sup\{\int_{E}\Div(\phi),\phi\in C^{1}_{0}(\mathring{E},\bR^{n}),\|\phi\|_{C^{0}}\leq 1\},
    \]
    we call $P(E)$ the perimeter of $E$ in $N$.
\end{defi}

We always denote $\partial^{*}E$ by the reduced boundary of $E$. In particular, we remark that  when $N$ is a compact manifold with boundary, the perimeter $P(E)$ does not charge the boundary. For more details of finite perimeter set, we refer to \cite{Maggi}.
\subsection{Free boundary unequally warped $\mu$-bubbles}
 Let $(N,\partial N)$ be an $n$-dimensional compact manifold with piecewise smooth boundary. As shown in Figure 1,  we let $\partial N =\partial_{0}N\cup\partial_{-}N\cup\partial_{+}N$, where $\partial_{-}N$ and $\partial_{+}N$ are disjoint and the interior angles between $\partial_{+}N$, $\partial_{-}N$ and $\partial_{0}N$ are less than or equal to $\frac{\pi}{2}$. This angle condition will force the minimizer of following functional $\cA(\Omega)$ stay away from the corner to have enough regularity. 
In particular, when $\partial_{+}N$ and $\partial_{-}N$ intersect the boundary $\partial_{0}N$ with angles more than $\pi/2$, the angles can be adjusted to be no greater than $\pi/2$ by perturbing $N$ near arbitrary small neighborhood according to \cite[Lemma 4.3]{Chen-Hong2025}.} For convenience of readers, we give the following Lemma \ref{lemm-7}, which will be used in the proof of Proposition \ref{prop-8}.
\begin{figure}[H]
    \centering
\includegraphics[width=0.5\textwidth]{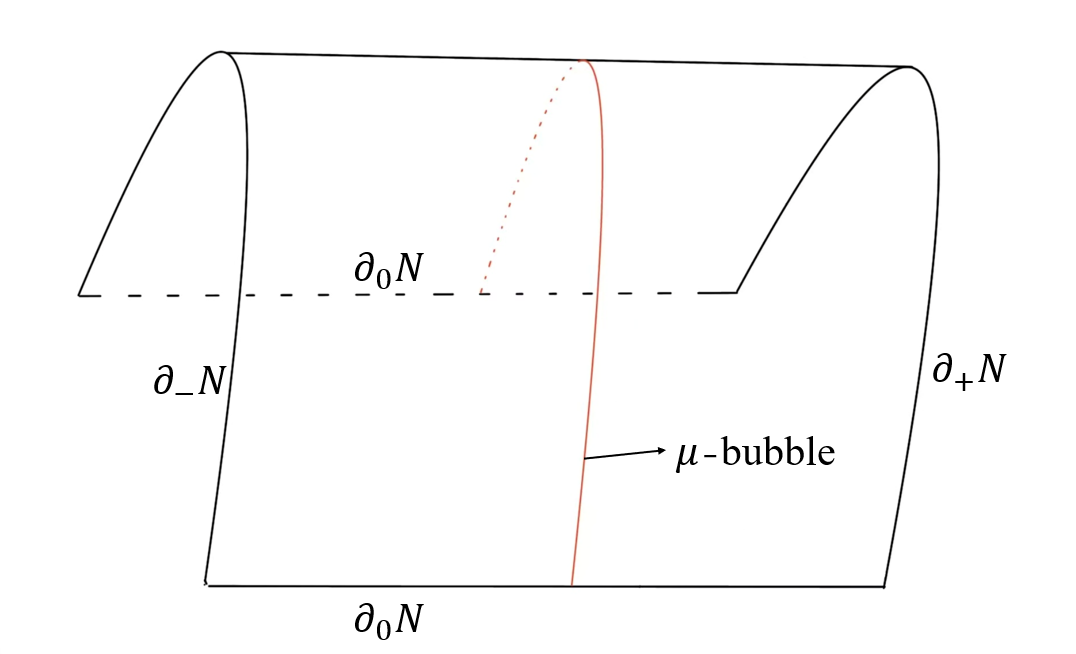}
    \caption{Free boundary $\mu$-bubble in $N$}
    \label{fig:my_label}
\end{figure}

{\color{red}\begin{lemm}\label{lemm-7}(\cite[Lemma 4.3]{Chen-Hong2025})
  Let $(N,\partial N)$ be a smooth compact  manifold with non-empty boundary. Let $\rho:N\to \bR$ be a smooth function with $\Lip(\rho)\leq D$. For any $a>0$ and $\epsilon>0$, if $\partial(\{\rho(x)<a\})$ is a compact set, then there exists a small constant $\epsilon^{'}>0$ and a smooth function $\tilde{\rho}$ on $N$ such that the dihedral angle between $\partial(\{\tilde{\rho}<a\})\setminus\partial N$ and $\partial N\cap\partial(\{\tilde{\rho}<a\})$ is less than or equal to $\frac{\pi}{2}$, $\Lip(\tilde{\rho})\leq D+\epsilon^{'}$ and $\tilde{\rho}(x)=\rho(x)$ for $x\in \{\rho(x)\leq a-\epsilon\}$.
\end{lemm}
}
We fix a smooth function $w>0$ in $N$ and a smooth function $h$ in $N\backslash(\partial_{-}N\cup\partial_{+}N)$, with $h\rightarrow\pm\infty$ on $\partial_{\pm}N$. We choose a regular value $c_{0}$ of $h$ on $N\setminus(\partial_{-}N\cup\partial_{+}N)$ and take $\Omega_{0}=h^{-1}((c_{0},\infty))$ as the reference set, and consider the following area functional:
$$\cA(\Omega):=\int_{\partial^{*}\Omega}w^{\theta}d\cH^{n-1}-\int_{N}(\chi_{\Omega}-\chi_{\Omega_{0}})hw^{\alpha}d\cH^{n},$$
for all Caccioppoli sets $\Omega$ with $\Omega\Delta\Omega_{0}\subset\subset \mathring{N}$, where $\partial^{*} \Omega$ denotes the reduced boundary of $\Omega$, $0\leq\theta\leq \frac{n-1}{n-2},\alpha=\frac{2\theta}{n-1}$, and the reference set $\Omega_{0}$ with smooth boundary satisfies$$\partial\Omega_{0}\subset\mathring{N},\quad\quad{\color{red}\partial_{+}N\subset\overline{\Omega}_{0}}.$$
If there exists an $\Omega$ that minimizes $\mathcal{A}$ in this class, then we call $\partial\Omega$ a free boundary unequally warped $\mu$-bubble.

{\color{red}\begin{prop}\label{prop-8}
There exists a minimizer $\Omega$ for $\cA$ such that $\Omega\Delta\Omega_{0}$ is compactly contained in $\mathring{N}\cup\partial_{0}N$. The singular set of $\partial\Omega$ has Hausdorff dimension at most $n-8$, whose boundary intersects with $\partial_{0}N$ orthogonally.
\end{prop}
\begin{proof}
    The proof is essentially the same as that in \cite{general soap bubbles}, with slight modifications. Let $\Omega$ be a Caccioppoli set in $N$ such that $\Omega\Delta\Omega_{0}\subset\subset\mathring{N}$. By a standard approximation argument we can assume that $\Omega$ has smooth boundary. We will prove that, by adding to $\Omega$ a neighborhood of $\partial_{+}N$, and subtracting from it a neighborhood of $\partial_{-}N$, the functional $\cA$ always decreases. Hence, we construct small neighborhoods of $\partial_{\pm}N$ in $N$. For $\tau>0$, denote $\Omega^{t}_{\pm}$ by the distance-$\tau$ neighborhood of $\partial_{\pm}N$ in $N$. Choosing $\tau$ sufficiently small, $\Omega^{\tau}_{\pm}$ has a foliation $\{S^{\rho}_{\pm}\}_{\rho\in[0,\tau]}$ by smooth equidistant hypersurfaces to $\partial_{\pm}N$. Denote by $\upsilon$ the unit normal vector field of $\{S^{\rho}_{\pm}\}$ defined in this $\Omega^{\tau}_{\pm}$, pointing into $N$ along $\partial_{\pm}N$. Let $\tau>0$ be sufficiently small so that 
    \[
    hw^{\alpha}>\Div(w^{\theta}\upsilon)\quad in\quad\Omega^{\tau}_{\pm},\quad hw^{\alpha}<-\Div(w^{\theta}\upsilon)\quad in \quad\Omega^{\tau}_{\pm}.
    \]
    We compute that 
    \begin{align*}
        \cA(\Omega\cup\Omega^{\tau}_{+})-\cA(\Omega)&=\int_{\partial^{*}(\Omega\cup\Omega^{\tau}_{+})}w^{\theta}d\cH^{n-1}-\int_{\partial\Omega}w^{\theta}d\cH^{n-1}-\int_{\Omega^{\tau}_{+}\setminus\Omega}hw^{\alpha}d\cH^{n}\\
&=\int_{\partial\Omega^{\tau}_{+}\setminus\Omega}w^{\theta}d\cH^{n-1}-\int_{\partial\Omega\cap\Omega^{\tau}_{+}}w^{\theta}d\cH^{n-1}-\int_{\Omega^{\tau}_{+}\setminus\Omega}hw^{\alpha}d\cH^{n}\\
&<\int_{\partial\Omega^{\tau}_{+}\setminus\Omega}w^{\theta}d\cH^{n-1}-\int_{\partial\Omega\cap\Omega^{\tau}_{+}}w^{\theta}d\cH^{n-1}-\int_{\Omega^{\tau}_{+}\setminus\Omega}\Div(w^{\theta}\upsilon)d\cH^{n}.
\end{align*}
On the other hand, we note that
\begin{align*}
    \int_{\Omega^{\tau}_{+}\setminus\Omega}\Div(w^{\theta}\upsilon)d\cH^{n}&=\int_{\partial\Omega^{\tau}_{+}\setminus\Omega}(\upsilon\cdot\nu)w^{\theta}d\cH^{n-1}-\int_{\partial\Omega\cap\Omega^{\tau}_{+}}(\upsilon\cdot\nu)w^{\theta}d\cH^{n-1}\\
    &\quad+\int_{\partial N\cap(\Omega^{\tau}_{+}\setminus\Omega)}w^{\theta}\bangle{\upsilon,\nu_{\partial_{N}}}d\cH^{n-1}\\
    &\geq\int_{\partial\Omega^{\tau}_{+}\setminus\Omega}w^{\theta}d\cH^{n-1}-\int_{\partial\Omega\setminus\Omega^{\tau}_{+}}w^{\theta}d\cH^{n-1}+\int_{\partial N\cap(\Omega^{\tau}_{+}\setminus\Omega)}w^{\theta}\cos{\beta} d\cH^{n-1}.
\end{align*}
We note that when $\tau$ is sufficiently small, the interior angle $\beta$ between $\partial\Omega^{\tau}_{+}$ and $\partial N$ is at most $\frac{\pi}{2}$, which guarantees that the term $\langle\upsilon,\nu_{\partial N}\rangle=\cos\beta\geq 0$. Therefore, we have $\mathcal{A}(\Omega\cup\Omega^{\tau}_{+})<\mathcal{A}(\Omega)$. By an analogous calculation, we can also conclude that $\mathcal{A}(\Omega\setminus\Omega^{\tau}_{-})<\mathcal{A}(\Omega)$.

Next, we consider the infimum of $\cA$ among $\cC$, where $\cC$ is the collection of Caccioppoli sets that contain $\Omega^{\tau}_{+}$ and are disjoint from $\Omega^{\tau}_{-}$. Since $|hw^{\alpha}|<C_{1}$ in $N\setminus(\Omega^{\tau}_{+}\cup\Omega^{\tau}_{-})$, we conclude that $\cA(\Omega)>-C_{1}\cH^{n}(N)$ for all $\Omega\in\cC$. Hence, the infimum $I_{\cA}=\inf\{\cA(\Omega):\Omega\in\cC\}$ exists. Choose a minimizing sequence $\{\Omega_{k}\}\subset\cC$ such that $\cA(\Omega_{k})\to I_{\cA}$. It follows that $\cH^{n-1}(\partial\Omega_{k})<C(I_{\cA}+C_{1}\cH^{n}(N))$. By the BV-compactness theorem, after passing to a subsequence, the sets $\Omega_{k}$ converge to a Caccioppoli set $\Omega$. Consequently, $\Omega$ is a minimizer of $\cA$, and its boundary $\partial\Omega$ is a hypersurface in $N$ intersecting $\partial_{0}N$ orthogonally, as guaranteed by the first variation formula presented below. For the regularity of the minimizer, we refer the reader to \cite[Theorem 1.1]{Chodosh-Edelen-Li}.
\end{proof}
 }


 {\color{red}Having established the existence of an unequally warped free boundary $\mu$-bubble, we now compute the first and second variation formula of the functional $\cA(\Omega)$. The proof of the following lemma \ref{first-area} is given in Appendix A.

 Assume that $\Omega$ is the minimizer of $\cA$ and $\Sigma=\partial\Omega$. Let $\secondfund_{\Sigma}$, $H_{\Sigma}$ denote by the second fundamental form and mean curvature of $\Sigma$, respectively.Let $\nu_{\partial \Sigma}$ denote the outward unit normal vector field of $\partial\Sigma$ in $\Sigma$, $A_{\partial N}$ be the second fundamental form of $\partial N$ with respect to $\nu_{\partial N}$, where $\nu_{\partial N}$ is the outward unit normal vector field of $\partial_{0}N$ in $N$.}

 \begin{lemm}\label{first-area}
Assume that $\Omega$ is a minimizer of $\cA$, and the corresponding normal variational vector field is $\phi\nu_{\Sigma}$ at $t=0$, then we have the first variation formula
\begin{align*}
\frac{d}{dt}\big|_{t=0}\cA(\Omega_{t})&=\int_{\partial\Omega}\theta w^{\theta-1}\bangle{\nabla^{N}w,\nu_{\Sigma_{t}}}\phi+w^{\theta}H_{\Sigma}\phi-hw^{\alpha}\phi+\int_{\partial \Sigma}w^{\theta}\bangle{\phi\nu_{\Sigma},\nu_{\partial\Sigma_{t}}}.
\end{align*}
In particular, $\Sigma=\partial\Omega$ satisfies
\[
H_{\Sigma}+\theta w^{-1}\bangle{\nabla^{N}w,\nu_{\Sigma}}-hw^{\alpha-\theta}=0~~on~~\Sigma,~~\nu_{\Sigma}\perp T_{x}\partial N~~for ~~x\in \partial\Sigma\subset\partial_{0}N.
\]
And the second variation formula is 
\begin{align*}
  \frac{d^2}{dt^{2}}\big|_{t=0}\cA(\Omega_{t})&=\int_{\Sigma}\big[-\Delta_{\Sigma}\phi-|\secondfund_{\Sigma}|^2\phi-\Ric_{N}(\nu_{\Sigma},\nu_{\Sigma})\phi-\theta w^{-2}\bangle{\nabla^{N}w,\nu_{\Sigma}}^2\phi\\
  &\quad+\theta w^{-1}\phi(\Delta_{N}w-\Delta_{\Sigma}w-H_{\Sigma}\bangle{\nabla^{N}w,\nu_{\Sigma}})-\theta w^{-1}\bangle{\nabla^{\Sigma}w,\nabla^{\Sigma}\phi}\\
  &\quad-\phi\bangle{\nabla^{N}h,\nu_{\Sigma}}w^{\alpha-\theta}+(\theta-\alpha)w^{\alpha-\theta-1}h\phi\bangle{\nabla^{N}w,\nu_{\Sigma}}\big]w^{\theta}\phi\\
&\quad+\int_{\partial\Sigma}w^{\theta}\phi\frac{\partial \phi}{\partial\nu_{\partial \Sigma}}-A_{\partial N}(\nu_{\Sigma},\nu_{\Sigma})\phi^2w^{\theta}.
\end{align*}
\end{lemm}

\subsection{Free boundary unequally warped isoperimetric profile}
Next, we give some basic notions and regularity results about isoperimetric profiles, we refer the readers to \cite{Manuel-Ritore} for more details.
\begin{defi}
  Let $(M^n,g)$ be a {\color{red}connected} compact Riemannian manifold with boundary. The isoperimetric profile of $M$ is the function $I_{M}$ that assigns, to each $\upsilon\in(0,|M|)$, the value $$I_{M}(\upsilon)=\inf\{P(E): E\enspace is \enspace measurable, |E|=\upsilon\}.$$  $I$ denotes by the isoperimetric profile of $M$ in this article.
\end{defi}
\begin{defi}
  Let $(M,g)$ be a Riemannian manifold. We say that a set $E\subset M$ is isoperimetric or is an isoperimetric region if $$P(E)=I_{M}(|E|).$$
  If $|E|=\upsilon$, then we say that $E$ is an isoperimetric region of volume $\upsilon$.
\end{defi}
About its existence, we find some results from \cite{Manuel-Ritore} as follows.
\begin{theo}(\cite[Theorem 9.3]{Manuel-Ritore})
  Let $M$ be a compact Riemannian manifold with smooth boundary. Then:
\begin{itemize}
      \item Isoperimetric sets exist on $M$ for any volume $0<\upsilon<\Vol(M)$.
      \item The isoperimetric profile $I_{M}$ is continuous.
      \item $I_{M}>0$ on $(0, \Vol(M)).$
\end{itemize}
\end{theo}
\begin{rema}
A hypersurface $\Gamma\subset M$ satisfying $\partial \Gamma\subset\partial M$ that separates $\Omega$ into two sets is called an interface.
If a smooth interface $\Gamma$ separates $M$ into two sets, the relative perimeter of each of these sets is the area of the interface. This means that there are no contributions to the relative perimeter from pieces in $\partial M$. In addition, the critical points of the area functional are hypersurfaces that meets $\partial M$ along $\partial \Gamma$ in {\color{red}an} orthogonal way.
\end{rema}

Regarding regularity, we have the following result given by \cite{Manuel-Ritore}, which follows from Giusti \cite{Giusti}; Gonzalez et al.\cite{Gonzalez}; Bombieri \cite{Bombieri};
Gr\"{u}ter \cite{Gruter}; and Morgan \cite{Morgan}, see Proposition 2.3 in Bayle and Rosales \cite{Bayle}.
\begin{theo}(\cite[Theorem 9.4]{Manuel-Ritore})\label{regu-1}
    Let $E$ be a measurable set of finite volume minimizing perimeter under a volume constraint in  $M$ with smooth boundary. Then:
\begin{itemize}
  \item  {\color{red}If $n\leq 7$, then the boundary $S=cl(\partial E\cap M)$ of $E$ is a smooth hypersurface whose boundary lies in $\partial M$.}
 \item  If $n>7$, then the boundary of $cl(\partial E\cap M)$ is the union of of a smooth hypersurface $S$ and a closed singular set $S_{0}$ of Hausdorff dimension at most $n-8$.
\end{itemize}
\end{theo}
We now turn to the unequally weighted isoperimetric profile.
Let $n\geq 3$, and $(M^n,g)$ be a compact Riemannian manifold with weakly convex boundary. Let $0\leq\theta\leq\frac{n-1}{n-2}$, and set $$\alpha:=\frac{2\theta}{n-1}.$$ For an open set $E\subset M$ with smooth boundary, we can define unequally  weighted area and volume functional by$$A(E)=\int_{\partial^{*} E}w^{\theta}\enspace\text{and}\enspace V(E)=\int_{E}w^{\alpha},$$
where $w$ satisfies$$-\Delta_{M} w\geq \theta^{-1}((n-1)\lambda-\Lambda_{\Ric})w,\enspace in \enspace M, $$and$$\bangle{\nabla^{M}w,\eta}=0,\enspace on\enspace \partial M.$$
Let $V_{0}:=\int_{M}w^{\alpha}\in(0,\infty)$, and define the unequally weighted isoperimetric profile
\begin{align}\label{isoprofile}
I(\upsilon):=\inf\big\{\int_{\partial^{*}E}w^\theta:E\subset\subset M \enspace\text{has finite perimeter, and}\int_{E}w^{\alpha}=\upsilon\},
\end{align}
for all $\upsilon\in[0,V_{0}).$
\begin{rema}
We note that $I$ is continuous. On the one hand, by the compactness theory for Caccioppoli sets and the lower semi-continuity, we have $\lim\inf\limits_{\upsilon\to\upsilon_{0}}I(\upsilon)\geq I(\upsilon_{0})$. On the other hand, we will show that there exists a continuous upper barrier function for $I$ at $\upsilon_{0}$ for any $\upsilon_{0}\in(0,V_{0})$, we also have $\lim\sup\limits_{\upsilon\to\upsilon_{0}}I(\upsilon)\leq I(\upsilon_{0})$.
\end{rema}
Next, we compute the first and second variations of the functionals $A$ and $V$. Let $E_{t}$ be a smooth family of open sets with a smooth boundary whose variational vector field along $\Gamma=\partial E\cap M=\partial E_{0}\cap M$ is $\varphi \nu_{\Gamma}$, where $\nu_{\Gamma}$ is the outward unit normal vector field along $\Gamma$. We denote $\secondfund_{\Gamma}$ and $H_{\Gamma}$ by the second fundamental form of $\Gamma$ with respect to $\nu_{\Gamma}$ and the scalar mean curvature of $\Gamma$, respectively. We denote $A_{\partial\Gamma }$ by the second fundamental form of $\partial\Gamma$ with respect to $\nu_{\partial\Gamma}$, where $\nu_{\partial\Gamma}$ is the unit outward vector field of $\partial\Gamma$. Therefore, we can also get following results by similar computations as \cite{Antonelli-Xu}.
\begin{lemm}[The first and second variational formulas of the isoperimetric profile]\label{1stvar}

\begin{align*}
\frac{d}{dt}\big|_{t=0}V(E_{t})&=\int_{\Gamma} w^{\alpha}\varphi,\\
\frac{d}{dt}\big|_{t=0}A(E_{t})&=\int_{\Gamma}w^{\theta}\varphi(H_{\Gamma}+\theta w^{-1}\bangle{\nabla^{M}w,\nu_{\Gamma}}).\\
\frac{d^2}{dt^2}\big|_{t=0}V(E_{t})&=\int_{\Gamma}(H_{\Gamma}+\alpha w^{-1}\bangle{\nabla^{M}w,\nu_{\Gamma}})w^{\alpha}\varphi^2+w^{\alpha}\varphi\bangle{\nabla^{M}\varphi,\nu_{\Gamma}}.\\
\frac{d^2}{dt^2}\big|_{t=0}A(E_{t})&=\int_{\Gamma}(-\Delta_{\Gamma}\varphi-\Ric_{M}(\nu_{\Gamma},\nu_{\Gamma})\varphi-|\secondfund_{\Gamma}|^2\varphi-\theta w^{-2}\bangle{\nabla^{M}w,\nu_{\Gamma}}^2\varphi)w^{\theta}\varphi\\
&\quad+[\theta w^{-1}(\Delta_{M} w-\Delta_{\Gamma}w-H_{\Gamma}\bangle{\nabla^{M}w,\nu_{\Gamma}})\varphi-\theta w^{-1}\bangle{\nabla_{\Gamma}w,\nabla_{\Gamma}\varphi}]w^{\theta}\varphi\\
&\quad+(\theta w^{\alpha-1}\bangle{\nabla^{M}w,\nu}\varphi^2+w^{\alpha}\varphi\bangle{\nabla^{M}\varphi,\nu}+H_{\Gamma}w^{\alpha}\varphi^2)w^{\theta-\alpha}(H_{\Gamma}+\theta w^{-1}\bangle{\nabla^{M}w,\nu})\\
&\quad+\int_{\partial\Gamma}w^{\theta}\varphi\bangle{\nabla^{M}\varphi,\nu_{\partial\Gamma}}-A_{\partial M}(\nu_{\Gamma},\nu_{\Gamma})\varphi^2w^{\theta}.
\end{align*}
\end{lemm}
{\color{red}\begin{proof}
The proof is basically the same as the computations of \cite[Lemma 2]{Antonelli-Xu} except for boundary integrals. Regarding the boundary terms, the computation is similar with Lemma \ref{first-area}, we refer the reader to Appendix A.
\end{proof}}
\section{diameter and volume estimates for $3\leq n\leq 7$}
Considering the singularities of the minimizers of $\cA(\Omega)$ and $I(\upsilon)$ that will occur when $n\geq 8$, we first prove the case of $3\leq n\leq 7$, and we will deal with the singular case of $n\geq 8$ in Section 4 alone.
\subsection{Diameter estimates}
Let $M^n$ be an $n$-dimensional compact manifold with $A_{\partial M}\geq 0$, following the method of Antonelli-Xu \cite{Antonelli-Xu}, we can obtain the diameter estimates in the sense of spectrum condition.
\begin{theo}\label{diamter theo}
  Let $M^n$, $n\geq 3$, be a compact connected manifold with weakly convex boundary $\partial M$, and let $0\leq \theta\leq \frac{n-1}{n-2}$, $\lambda>0$. We denote by $\Lambda_{\Ric}(x):=\inf_{\upsilon\in T_{p}M,|\upsilon|=1}\Ric_{x}(\upsilon,\upsilon)$ the smallest eigenvalue of the Ricci tensor. If there exists a positive function $w\in C^{\infty}(M)$ that satisfies:
\begin{equation*}
  \begin{cases}
    \theta\Delta_{M}w\leq \Lambda_{\Ric} w-(n-1)\lambda w,\enspace\enspace in \enspace $M$.\\
    \frac{\partial w}{\partial \eta}=0.\enspace\enspace on \enspace \partial M.\\
  \end{cases}
\end{equation*}
Then we have the diameter estimate: $\diam(M)\leq\frac{\pi}{\sqrt{\lambda}}(\frac{\max{w}}{\min{w}})^{\frac{n-3}{n-1}\theta}$.
\end{theo}
\begin{proof}
We can use the same method from \cite{Antonelli-Xu}. Suppose by contradiction that the above diameter estimate does not hold, then there is a $\epsilon>0$ such that
\begin{align}\label{diam}
 \diam(M)>\frac{\pi}{\sqrt{\lambda}}\cdot(\frac{\max(w)}{\min(w)})^{\frac{n-3}{n-1}\theta}\cdot({1+\epsilon})^2+2\epsilon.
\end{align}\label{diam}
Let us fix a point $p\in \partial M$ and take $\Omega_{+}:=B_{\epsilon}(p)$, and let $d:M\setminus\Omega_{+}\rightarrow\bR$ be a smoothing of $d(\cdot,\partial\Omega_{+})$ such that
$$d\big|_{\partial\Omega_{+}}=0,\quad\quad\big|\nabla d\big|\leq 1+\epsilon,\quad\quad d\geq\frac{d(\cdot,\partial\Omega_{+})}{1+\epsilon}.$$
We let
$$h(x):=\sqrt{\frac{(n-1)^2\lambda}{\max(w)^{\frac{6-2n}{n-1}\theta}}}\cot(\frac{1}{1+\epsilon}\sqrt{({\frac{\max(w)}{\min(w)}})^{\frac{6-2n}{n-1}\theta}\lambda}~d(x)),$$
First, $h$ is a smooth function on $M$ and
\begin{align}\label{ine-h}
|\nabla h|<\frac{\max(w)^{\frac{6-2n}{n-1}\theta}}{(n-1)\min(w)^{\frac{3-n}{n-1}\theta}}h^2+\frac{(n-1)\lambda}{\min(w)^{\frac{3-n}{n-1}\theta}},
\end{align}
then $$|\nabla h|w^{\frac{3-n}{n-1}\theta}\leq|\nabla h|\min(w)^{\frac{3-n}{n-1}\theta}<\frac{h^2w^{\frac{6-2n}{n-1}\theta}}{n-1}+(n-1)\lambda.$$
Let $$\cO:=\{d>\frac{2(1+\epsilon)\pi}{\sqrt{({\frac{\max(w)}{\min(w)}})^{\frac{6-2n}{n-1}\theta}\lambda}}\}\supset\{d(\cdot,p)>\epsilon+\frac{2(1+\epsilon)^2\pi}{\sqrt{({\frac{\max(w)}{\min(w)}})^{\frac{6-2n}{n-1}\theta}\lambda}}\}\neq\emptyset.$$
Set $\Omega_{-}:=M\backslash\bar{\cO}$, then we have found two domains $\Omega_{+}\subset\subset \Omega_{-}\subset\subset M$ and $h(x)\in C^{\infty}(\Omega_{-}\backslash\overline{\Omega_{+}})$ which satisfies
\begin{align}\label{eq1}
\lim\limits_{x\rightarrow\partial\Omega_{+}}h(x)=+\infty,\quad\quad \lim\limits_{x\rightarrow\partial\Omega_{-}}h(x)=-\infty.
\end{align}
{\color{red} According to the preliminaries in Section 2, we know there exists a free boundary $\mu$-bubble $\Omega$ that minimizes functional $\cA(\Omega_{t})$.}
By Lemma \ref{first-area}, we have
\begin{align*}
  &\frac{d^2}{dt^2}\big|_{t=0}\cA(\Omega_{t})\\
  &=\int_{\Sigma}[-\Delta_{\Sigma}\phi-|\secondfund_{\Sigma}|^2\phi-\Ric_{M}(\nu_{\Sigma},\nu_{\Sigma})\phi-\theta w^{-2}\bangle{\nabla^{M}w,\nu_{\Sigma}}^2\phi\\
  &\quad+\theta w^{-1}\phi(\Delta_{M}w-\Delta_{\Sigma}w-H_{\Sigma}\bangle{\nabla^{M}w,\nu_{\Sigma}}-\theta w^{-1}\bangle{\nabla^{\Sigma}w,\nabla^{\Sigma}\phi}\\
  &\quad-\phi\bangle{\nabla^{M}h,\upsilon_{\Sigma}}w^{\alpha-\theta}+(\theta-\alpha)w^{\alpha-\theta-1}h\phi\bangle{\nabla^{M}w,\nu_{\Sigma}}]w^{\theta}\phi
+\int_{\partial\Sigma}w^{\theta}\phi\frac{\partial \phi}{\partial\nu_{\partial \Sigma}}-A_{\partial M}(\nu_{\Sigma},\upsilon_{\Sigma})\phi^2w^{\theta}.
\end{align*}
Let $\phi=w^{-\theta}$, then
\begin{align*}
  \int_{\Sigma}-\Delta_{\Sigma}\phi & =\theta\int_{\partial\Sigma}w^{-\theta-1}\bangle{\nabla^{\Sigma}w,\nu_{\partial\Sigma}}.
\end{align*}
On the one hand,
\begin{align*}
&-\theta\int_{\Sigma}w^{-1}w^{-\theta}\Delta_{\Sigma}w+w^{-1}\bangle{\nabla^{\Sigma}w,\nabla^{\Sigma}(w^{-\theta})}\\
&=-\theta\int_{\Sigma}w^{-\theta-1}\Delta_{\Sigma}w+\bangle{\nabla^{\Sigma}w,\nabla^{\Sigma}(w^{-\theta-1})}+\theta\int_{\Sigma}\bangle{\nabla^{\Sigma}w,\nabla^{\Sigma}(w^{-1})}w^{-\theta}\\
&=-\theta\int_{\partial\Sigma}\bangle{\nabla^{\Sigma}w,\nu_{\partial\Sigma}}w^{-\theta-1}-\theta\int_{\Sigma}|\nabla^{\Sigma}w|^2w^{-2-\theta}\\
&\leq-\theta\int_{\partial\Sigma}\bangle{\nabla^{\Sigma}w,\nu_{\partial\Sigma}}w^{-\theta-1}.\\
\end{align*}
Since $$|\secondfund_{\Sigma}|^2\geq\frac{1}{n-1}H_{\Sigma}^2=\frac{1}{n-1}(h-\theta w^{-1}\bangle{\nabla^{M}w,\nu_{\Sigma}})^2,$$ and $$\bangle{\nabla^{\Sigma}w,\nu_{\partial\Sigma}}=\bangle{\nabla^{M}w,\nu_{\partial\Sigma}}=0,\enspace A_{\partial M}\geq 0\enspace \text{on $\partial\Sigma$},$$ where $A_{\partial M}$ is the second fundamental form of $\partial M$ with respect to $\eta$, which is guaranteed by the weakly convexity of $\partial M$.
Combining these consequences with non-negativity of the second variation, we obtain that

\begin{align*}
  0 &\leq\int_{\Sigma}-|\secondfund_{\Sigma}|^2w^{-\theta}-\Ric_{M}(\nu_{\Sigma},\upsilon_{\Sigma})w^{-\theta}-\theta w^{-2}\bangle{\nabla^{M}w,\nu_{\Sigma}}^2\\
  &\quad+\theta w^{-1-\theta}(\Delta_{M}w-H_{\Sigma}\bangle{\nabla^{M}w,\nu_{\Sigma}})-\bangle{\nabla^{M}h,\nu_{\Sigma}}w^{\alpha-2\theta}+(\theta-\alpha)hw^{\alpha-2\theta-1}\bangle{\nabla^{M}w,\nu_{\Sigma}}\\
  &{\color{red}\leq\int_{\Sigma}} w^{-\theta}\big[-\frac{H_{\Sigma}^2}{n-1}-(n-1)\lambda-\theta(w^{-1}\bangle{\nabla^{M}w,\nu_{\Sigma}})^2-\theta H_{\Sigma}(w^{-1}\bangle{\nabla^{M}w,\nu_{\Sigma}})+|\nabla h|w^{\alpha-\theta}\\
  &\quad+(\theta-\alpha)hw^{\alpha-\theta}(w^{-1}\bangle{\nabla^{M}w,\nu_{\Sigma}})\big].
\end{align*}
Setting $X=hw^{\alpha-\theta}$, $Y=w^{-1}\bangle{\nabla^{M}w,\nu_{\Sigma}}$, then $H_{\Sigma}=X-\theta Y$, we have
\begin{align*}
0&\leq\int_{\Sigma}w^{-\theta}[-\frac{X^2}{n-1}+\frac{2\theta}{n-1}XY-\frac{\theta^2}{n-1}Y^2-(n-1)\lambda-\theta Y^2-\theta(X-\theta Y)Y+|\nabla h|w^{\alpha-\theta}+(\theta-\alpha)XY]\\
&\leq\int_{\Sigma}w^{-\theta}[-\frac{X^2}{n-1}-(n-1)\lambda+|\nabla h|w^{\alpha-\theta}+(\frac{2\theta}{n-1}-\alpha)XY+(\frac{n-2}{n-1}\theta^2-\theta)Y^2].
\end{align*}
If we set $\alpha=\frac{2\theta}{n-1}$, for $0\leq \theta\leq \frac{n-1}{n-2}$.
But $h$ satisfies $$|\nabla h|w^{\alpha-\theta}<\frac{h^2w^{2\alpha-2\theta}}{n-1}+(n-1)\lambda,$$ this is a contradiction.
\end{proof}


\subsection{Volume comparison}

\subsubsection{Differential inequality in the barrier sense}
Fix a $\upsilon_{0}\in(0,V_{0})$. Let $E$ be a weighted isoperimetric hypersurface with free boundary in $M$ for problem $I(\upsilon_{0})$. From now on, we fix $\varphi=w^{-\theta}$.
We notice that $V(t):=V(E_{t})$ is a smooth function. By the first variation information, we have$$V'(0)=\frac{d}{dt}\big|_{t=0}V(E_{t})=\int_{\partial E}w^{\alpha-\theta}>0,$$ hence $V(t)$ is a strictly monotone in $t$ in a neighborhood of $\upsilon_{0}$. By the inverse function theorem, there is some small $\sigma>0$ and a smooth function $$t:(\upsilon_{0}-\sigma,\upsilon_{0}+\sigma)\longrightarrow\bR,$$ that is the inverse of $V(t)$.

Let $u:(\upsilon_{0}-\sigma,\upsilon_{0}+\sigma)\to\bR$ be defined by $u(\upsilon)=A(t(\upsilon))$. Note that $u(\upsilon_{0})=A(E_{0})=I(\upsilon_{0})$. Moreover, since $V(E_{t(\upsilon)})=\upsilon$, we have $A(\upsilon)\geq I(\upsilon)$ for all $\upsilon\in(\upsilon_{0}-\sigma,\upsilon_{0}+\sigma)$. Let primes denote the derivatives with respect to $\upsilon$ and dots denote the derivatives with respect to $t$.
\begin{lemm}\label{lemm4}
Let $(M^n,\partial M)$ be a complete, connected compact manifold with weakly convex boundary. Assume $w\in C^{\infty}(M)$ attains its minimal value on $\partial M$ and satisfies $\inf\limits_{p\in\partial M} w=1$, and it holds that
\begin{equation}\label{spec-1}
    \begin{cases}
    \theta\Delta_{M} w\leq w\Lambda_{\Ric}-(n-1)\lambda w, \enspace in \enspace M.\\
\frac{\partial w}{\partial \eta}=0,\enspace on\enspace \partial M. \\
    \end{cases}
\end{equation}
Suppose for fixed $\upsilon_{0}\in(0,V_{0})$, there exists a bounded set $E$ with finite perimeter, such that $\int_{E}w^{\alpha}=\upsilon_{0}$ and $\int_{\partial^{*}E}w^{\theta}=I(\upsilon_{0})$. Then $I$ satisfies $$I''I\leq-\frac{(I')^2}{n-1}-(n-1)\lambda$$ in the barrier sense at $\upsilon_{0}$.
\end{lemm}
\begin{proof}
We set $\Gamma=\partial E\cap M$, By the first and second derivatives of an inverse function:$$t'(\upsilon)=\frac{1}{\dot{\upsilon}}(t(\upsilon))\enspace and \enspace t''(\upsilon)=-\frac{\ddot{\upsilon}(t(\upsilon))}{\dot{\upsilon}(t(\upsilon))^3},$$ then we have $$t'(\upsilon_{0})=(\int_{\Gamma}w^{\alpha-\theta})^{-1}$$ and$$ t''(\upsilon_{0})=-(\int_{\Gamma}w^{\alpha-\theta})^{-3}\int_{\Gamma}(H_{\Gamma}+\alpha w^{-1}\bangle{\nabla^{M}w,\nu_{\Gamma}}w^{\alpha}\varphi^2+w^{\alpha}\varphi\bangle{\nabla^{M}\varphi,\nu_{\Gamma}})$$
 We note that $$\frac{d}{d\upsilon}A(t(\upsilon))=\dot{A}(t(\upsilon))t'(\upsilon)$$ and $$\frac{d^2}{d\upsilon^2}A(t(\upsilon))=\ddot{A}(t(\upsilon))(t'(\upsilon))^2+\dot{A}(t(\upsilon))t''(\upsilon).$$
 Since $E$ is a volume-constrained minimizer, then $\frac{d}{dt}|_{t=0}A(E_{t})=0$ whenever  $\frac{d}{dt}|_{t=0}V(E_{t})=0$. Therefore, by the definition of $A'(\upsilon)$, we conclude that $A'(\upsilon_{0})=w^{\theta-\alpha}(H_{\Gamma}+\theta w^{-1}\bangle{\nabla^{M}w,\nu_{\Gamma}})$ is constant.
 By Proposition \ref{1stvar}, choosing $\varphi=w^{-\theta}$, we have
 \begin{align*}
 \frac{d^2}{dt^2}\big|_{t=0}A(E_{t})&=\int_{\Gamma}(-\Delta_{\Gamma}(w^{-\theta})-\Ric_{M}(\nu_{\Gamma},\nu_{\Gamma})w^{-\theta}-|\secondfund_{\Gamma}|^2w^{-\theta}-\theta w^{-2-\theta}\bangle{\nabla^{M}w,\nu_{\Gamma}}^2\\
&\quad+\theta w^{-\theta-1}(\Delta_{M} w-\Delta_{\Gamma}w-H_{\Gamma}\bangle{\nabla^{M}w,\nu_{\Gamma}})-\theta w^{-1}\bangle{\nabla_{\Gamma}w,\nabla_{\Gamma}(w^{-\theta})}\\
&\quad+(H_{\Gamma}w^{\alpha-2\theta})w^{\theta-\alpha}(H_{\Gamma}+\theta w^{-1}\bangle{\nabla^{M}w,\nu_{\Gamma}})\\
&\quad+\int_{\partial\Gamma}\bangle{\nabla^{M}(w^{-\theta}),\nu_{\partial\Gamma}}-A_{\partial M}(\nu_{\Gamma},\nu_{\Gamma})w^{-\theta}.
 \end{align*}
 Using the weakly convexity of $\partial M$, integrating by parts and rearranging, we obtain that
 \begin{align*}
 \frac{d^2}{dt^2}\big|_{t=0}A(E_{t})&\leq\int_{\Gamma}-\Ric_{M}(\nu_{\Gamma},\nu_{\Gamma})w^{-\theta}-|\secondfund_{\Gamma}|^2w^{-\theta}+\theta w^{-\theta-1}\Delta_{M} w-\theta w^{-2-\theta}\bangle{\nabla^{M}w,\nu_{\Gamma}}^2\\
 &\quad-\theta H_{\Gamma}w^{-\theta-1}\bangle{\nabla^{M}w,\nu_{\Gamma}}+H_{\Gamma}w^{-\theta}(H_{\Gamma}+\theta w^{-1}\bangle{\nabla^{M}w,\nu_{\Gamma}})
 \end{align*}

 According to inequality (\ref{spec-1}) and set $X=w^{\alpha-\theta}A'(\upsilon_{0})$, $Y=w^{-1}\bangle{\nabla^{M}w,\nu_{\Gamma}}$, thus $H=X-\theta Y$. Using the trace inequality $|\secondfund_{\Gamma}|^2\geq H_{\Gamma}^2/(n-1)$, we have
 \begin{align*}
 \frac{d^2}{dt^2}\big|_{t=0}A(E_{t})&\leq\int_{\Gamma}-(n-1)\lambda w^{-\theta}+w^{-\theta}(-\frac{X^2}{n-1}+\frac{2\theta XY}{n-1}-\frac{\theta^2Y^2}{n-1}-\theta Y^2\\
 &\quad-\theta XY+\theta^2 Y^2+X^2-\theta XY).
 \end{align*}
 Therefore, we obtain
 \begin{align*}
 (\int_{\Gamma}w^{\alpha-\theta})^2\cdot{A}''(\upsilon_{0})&=(\ddot{A}(0)(t'(\upsilon_{0}))^2+\dot{A}(0)t''(\upsilon_{0}))(\int_{\Gamma}w^{\alpha-\theta})^2\\
 &\leq\int_{\Gamma}-(n-1)\lambda w^{-\theta}+w^{-\theta}(-\frac{X^2}{n-1}+\frac{2\theta XY}{n-1}-\frac{\theta^2Y^2}{n-1}-\theta Y^2\\
 &\quad-\theta XY+\theta^2 Y^2+X^2-\theta XY)\\
 &-(\int_{\Gamma}w^{\alpha-\theta})^{-1}(\int_{\Gamma} w^{\alpha-\theta}A'(\upsilon_{0}))\cdot(\int_{\Gamma}(X+\alpha Y-2\theta Y)w^{\alpha-2\theta})\\
 &=\int_{\Gamma}-(n-1)\lambda w^{-\theta}+w^{-\theta}\big(-\frac{X^2}{n-1}+\frac{2\theta XY}{n-1}-\frac{\theta^2Y^2}{n-1}-\theta Y^2\\
 &\quad-\theta XY+\theta^2 Y^2+X^2-\theta XY-X^2-\alpha XY+2\theta XY)\\
 &=\int_{\Gamma}w^{-\theta}\big[-\frac{X^2}{n-1}+(\frac{2\theta}{n-1}-\alpha)XY+(\frac{n-2}{n-1}\theta^2-\theta)Y^2]-(n-1)\lambda w^{\theta}\\
 &\leq\int_{\Gamma}-\frac{X^2}{n-1}w^{-\theta}-(n-1)\lambda w^{-\theta}\\
 &=\int_{\Gamma}-\frac{A'(\upsilon_{0})}{n-1}w^{2\alpha-3\theta}-(n-1)\lambda w^{-\theta},
 \end{align*}
 where we used the fact that
 $\alpha=\frac{2\theta}{n-1}$, and $0\leq \theta\leq\frac{n-1}{n-2}$. On the one hand, since $\alpha\leq\theta$, $w\geq 1$. Hence we have $$(\int_{\Gamma}w^{\alpha-\theta})^2\cdot{A}''(\upsilon_{0})\leq-\big(\frac{A'^{2}(\upsilon_{0})}{n-1}+(n-1)\lambda\big)\int_{\Gamma}w^{2\alpha-3\theta}.$$
 On the other hand, applying H\"{o}lder's inequality to  $A(\upsilon_{0})=\int_{\Gamma}w^{\theta}$ gives
 $$(\int_{\Gamma}w^{\alpha-\theta})^2\leq A(\upsilon_{0})\int_{\Gamma}w^{2\alpha-3\theta}.$$
 Putting them together, we have
 \begin{align}\label{barrier}
 A'(\upsilon_{0})A''(\upsilon_{0})\leq-\frac{A'^{2}(\upsilon_{0})}{n-1}-(n-1)\lambda.
 \end{align}
 Therefore, $A(\upsilon)$ is an upper barrier of $I(\upsilon)$ which satisfies the inequality (\ref{barrier}).
\end{proof}
We consider a power of $I$ and $A$ to simplify the corresponding differential inequality. Let $\cF(\upsilon)=I(\upsilon)^{\frac{n}{n-1}}$, combined with above lemma \ref{lemm4}, we have the following result.
\begin{prop}\label{prop6}
For any $\upsilon_{0}\in(0,V)$, there is a smooth function $U:(\upsilon_{0}-\sigma,\upsilon_{0}+\sigma)\longrightarrow\bR$ satisfying
\begin{itemize}
  \item $U(\upsilon_{0})=\cF(\upsilon_{0})$,
  \item $U(\upsilon)\geq\cF(\upsilon)$ for all $\upsilon\in(\upsilon_{0}-\sigma,\upsilon_{0}+\sigma)$,
  \item $U''(\upsilon_{0})\leq-\lambda nU(\upsilon_{0})^{\frac{2-n}{n}}$.
\end{itemize}
\end{prop}
\begin{proof}
We take $U(\upsilon)=A(\upsilon)^{\frac{n}{n-1}}$ as in Lemma \ref{lemm4}, and compute that $$U'(\upsilon)=\frac{n}{n-1}A(\upsilon)^{\frac{1}{n-1}}A'(\upsilon)$$ and
\begin{align*}
U''(\upsilon)&=\frac{n}{(n-1)(n-1)}A^{\frac{2-n}{n-1}}(\upsilon)A'^{2}(\upsilon)+\frac{n}{n-1}A^{\frac{1}{n-1}}(\upsilon)A''(\upsilon)\\
&=\frac{n}{n-1}A^{\frac{2-n}{n-1}}(\upsilon)(A'(\upsilon)A''(\upsilon)-(n-1)\lambda)+\frac{n}{n-1}A^{\frac{1}{n-1}}(\upsilon)A''(\upsilon)\\
&\leq-\lambda nU^{\frac{2-n}{n}}.
\end{align*}
\end{proof}
\subsubsection{Volume bound}
In this section, we will estimate the volume of manifold in the spectral sense. First, We begin by establishing an asymptotic volume expansion estimate for a small geodesic ball centered at a boundary point.
\begin{lemm}\label{lemm7}
Suppose that $M^n$ is a complete, connected compact manifold with weakly convex boundary, $w\in C^{\infty}(M)$ is positive. Assume $x\in \partial M$ satisfies $w(x)=\inf(w)=1$. Then, if $I$ is defined as in (\ref{isoprofile}), we have
\begin{align}\label{asympotic}
\lim\sup\limits_{\upsilon\to 0}\upsilon^{-\frac{n-1}{n}}I(\upsilon)\leq n\Vol(\bB^{n}_{+})^{\frac{1}{n}},
\end{align}
where $\bB^{n}_{+}$ is the unit half ball in $\bR^{n}$.
\end{lemm}
\begin{proof}
For a small $r_{0}$, the functions $V(r)=\int_{B(x,r)}w^{\alpha}$ and $A(r)=\int_{\partial B(x,r)}w^{\theta}$ are smooth and increasing in  $(0,r_{0})$, where the geodesic ball of radius $r$ is centered at $x\in\partial M$. We have the asymptotics
$$V(r)=\Vol(\bB^{n}_{+})r^n+O(r^{n+1}),$$
and$$A(r)=n\Vol(\bB^{n}_{+})r^{n-1}+O(r^n),$$
hence the function $A\circ V^{-1}(v)=n\Vol(\bB^{n}_{+})^{\frac{1}{n}}\upsilon^{\frac{n-1}{n}}+o(\upsilon^{\frac{n-1}{n}})$,
and $I(\upsilon)\leq A\circ V^{-1}(\upsilon)$.
\end{proof}
\begin{theo}
Let $V\in(0,\infty)$, and let $I:[0,V)\to\bR$ be a continuous function such that $I(0)=0$, and $I(\upsilon)>0$ for every $\upsilon\in(0,V)$. Assume that for some $\lambda>0$ we have $$I''I\leq-\frac{(I')^2}{n-1}-(n-1)\lambda,$$ and
\begin{align}\label{asymp-1}
\lim\sup\limits_{\upsilon\to 0^{+}}\upsilon^{-\frac{n-1}{n}}I(\upsilon)\leq n\Vol(\bB^{n}_{+})^{\frac{1}{n}}.
\end{align}
Then we have $\Vol(M)\leq\lambda^{-\frac{n}{2}}\Vol(\bS^{n}_{+})$, where $\bS^{n}_{+}$ denotes the unit half sphere in $\bR^{n+1}$.
\end{theo}
{\color{red}\begin{proof}
According to Lemma \ref{lemm4} and Proposition \ref{prop6}, we know that if we set $\cF(\upsilon)=I(\upsilon)^{\frac{n}{n-1}}$, then $\cF(\upsilon)$ satisfies a differential inequality in the barrier sense: $$\cF''(\upsilon)\leq-\lambda n\cF^{\frac{2-n}{n}}(\upsilon).$$
For every fixed $z>0$, we consider $f_{z}:(0,V_{z})\to \bR$ implicitly defined by
\begin{align}
f_{z}(z\int_{0}^{r}\beta(s)^{n-1}ds)&=z\beta(r)^{n-1},
\end{align}
where $\beta(s):[0,\frac{\pi}{\sqrt{\lambda}}]\to\bR$ is defined by $\beta(s)=\frac{1}{\sqrt{\lambda}}\sin(\sqrt{\lambda}r)$. We let $V_{z}=\frac{z}{\lambda^{\frac{n}{2}}}\frac{\Vol(\bS^{n})}{\Vol(\bS^{n-1})}$, then we have $f_{z}(0)=f_{z}(V_{z})=0$. If we set $\cF_{z}:=f_{z}^{\frac{n}{n-1}}$, a direct computation implies that
\begin{align}\label{asymp-2}
    \cF_{z}^{''}=-\lambda n\cF_{z}^{\frac{2-n}{n}}\quad\quad \text{on} ~(0,V_{z}).
\end{align}
Denoting by $\cF^{'}_{+}(0)$ the upper right-hand Dini derivative, by (\ref{asymp-1}), we have
$$\cF_{+}^{'}(0)\leq n^{\frac{n}{n-1}}\Vol(\bB^{n}_{+})^{\frac{1}{n-1}}.$$
On the other hand, we compute $(\cF_{z})^{'}_{+}(0)$ and obtain that
$$(\cF_{z})^{'}_{+}(0)=nz^{\frac{1}{n-1}}.$$
When $z=n\Vol(\bB^{n}_{+})=\Vol(\bS^{n-1}_{+})$, we conclude that $V_{z}=\lambda^{-\frac{n}{2}}\Vol(\bS^{n}_{+})$.

Suppose by contradiction that $V>\lambda^{-\frac{n}{2}}\Vol(\bS^{n}_{+})$, then there exists a $z^{'}>n\Vol(\bB^{n}_{+})$ such that $V_{z^{'}}<V$. Since $z^{'}>n\Vol(\bB^{n}_{+})$, we have
\begin{align}\label{asymp-3}
\cF_{+}^{'}(0)<(\cF_{z^{'}})^{'}_{+}(0)<+\infty.
\end{align}
Combining (\ref{asymp-1}), (\ref{asymp-2}) and (\ref{asymp-3}), we apply \cite[Th\'{e}or\`{e}me C.2.2]{Bayle2003}, \cite[Proposition A.3]{Pozzetta} to them,  and we finally get
$$\cF(\upsilon)\leq\cF_{z^{'}}(\upsilon)\quad\quad \text{for~all~}\upsilon\leq V_{z^{'}}.$$
We notice that $\cF(V_{z^{'}})\leq 0$ when  $\upsilon=V_{z^{'}}$, which is a contradiction with the fact that $I(\upsilon)>0$ on $(0,V_{z^{'}})\subset(0, V)$.
Therefore, since we normalized so that $\min\limits_{p\in \partial M} w=1$, we have $$\Vol(M)\leq\int_{M}w^{\alpha}=V\leq \lambda^{-\frac{n}{2}}\Vol(\bS^{n}_{+}).$$
\end{proof}}
{\color{red}\begin{rema}
If we do not require $w$ to attain its minimum on the boundary, we can always obtain a uniform volume bound for $M$. In fact, if the minimum of $w$ occurs in the interior of $M$, then following the closed case treated in \cite{Antonelli-Xu}, we can conclude that $\Vol(M)\leq \lambda^{-\frac{n}{2}}\Vol(\bS^{n})$.

\end{rema}}
We will give a characterization of rigidity for $M$ under the spectral lower bound on the Ricci curvature.
\subsection{Rigidity result}
Finally, we can establish a rigidity result following the method in \cite{Antonelli-Xu}. We can assume that $\lambda=1$ after rescaling.
Before giving the rigidity result, we recall a lemma derived by Wang in \cite{Wang21}.
\begin{lemm}[\cite{Wang21}]\label{rigidity-1}
    Let $(M^n,g)$ be a compact Riemannian manifold with weakly convex boundary and $\Ric\geq (n-1)g$. Then $$\Vol(M)\leq \Vol(\bS^{n})/2.$$
    Moreover, equality holds if and only if $M$ is isometric to $\bS^{n}_{+}$.
\end{lemm}
Combined with Lemma \ref{rigidity-1}, we obtain our rigidity results as follows.
\begin{theo}
Let $M^n$, $n\geq 3$, be a compact connected manifold with weakly convex boundary $\partial M$, and let $0\leq \theta\leq \frac{n-1}{n-2}$. Assume that there exist a positive function $w\in C^{\infty}(M)$ {\color{red}which} satisfies:
\begin{equation*}
  \begin{cases}
    -\theta\Delta w+\Lambda_{\Ric} w\geq (n-1)w,\enspace in \enspace M,\\
    \frac{\partial w}{\partial \eta}=0, \enspace on \enspace \partial M.\\
  \end{cases}
\end{equation*}
 If $w$ reaches its minimum in $\partial M$ and $\Vol{(M)}=\Vol(\bS^{n}_{+})$, then $w$ must be constant, and $M$ is isometric to the unit round hemisphere $\bS^{n}_{+}$.
\end{theo}
\section{singular case for $n\geq 8$ }
In this section, we discuss the singular case for the isoperimetric profile (or free boundary $\mu$-bubble) when $n\geq 8$. We extend the method of Bray et al. to the case of isoperimetric profile with free boundary. It's remarkable that we use the monotonicity formula about free boundary varifolds with generalized mean curvature, which is a generalization of Guang-Li-Zhou's paper.

\subsection{Control on Singular sets}
Our strategy for dealing with the singularities is also to control the area of isoperimetric profile around the singular sets. Unlike Antonelli-Xu's results, we need to consider the two interior and boundary singular sets cases.

In this section, our goal is to estimate the size of small neighborhoods around the singular sets such that we can carry out the flow outside these neighborhoods as in Section 3, hence complete the proof of the singular case.
Since the case of interior singularities has been proved in \cite{Antonelli-Xu}, we are devoted to controlling the area of the isoperimetric hypersurfaces around the singular sets that occur in the boundary. Therefore, we obtain the following local volume estimates about isoperimetric hypersurfaces.
\begin{lemm}\label{lemma-21}
    Let $\Sigma$ be an $(n-1)$-dimensional isoperimetric hypersurface with free boundary in $M$, for any $q\in\overline{\Sigma}$, the closure of $\Sigma$, we have the following uniform bound holds,
    $$\cH^{n-1}(B_{\rho}(q)\cap\Sigma)\leq C_{0}\rho^{n-1},$$
    for some positive constant $C_{0}$ depending only on $M$ and $\Sigma$.
\end{lemm}
If $B_{\rho}(q)\cap\partial\Sigma=\emptyset$, the above result has been proved in \cite[Lemma 3.2]{Bray-Gui-Liu-Zhang}. It suffices to prove that the above Lemma holds for the case that $B_{\rho}(q)\cap\partial\Sigma\neq \emptyset$. We first recall a monotonicity formula about stationary free boundary varifolds, which was established by Guang-Li-Zhou in \cite{Guang-Li-Zhou2018} when exploring the curvature estimates for stable free boundary minimal hypersurfaces. We observe that by slight modifications, this result is also true for free boundary varifolds with generalized mean curvature bounded above; we now only state the modified results here.
\begin{theo}(\cite[Theorem 3.4]{Guang-Li-Zhou2018})\label{free-vol}
   Assume that $M$ is an embedded $n$-dimensional submanifold in $\bR^{L}$ with the second fundamental form $A^{M}$ bounded by some constant $\Lambda_{1}>0$. Suppose that $N\subset M$ is a closed embedded $n$-dimensional submanifold and $V$ is a $k$-varifold with free boundary on $N$, and the generalized mean curvature $H$ of $V$ is bounded above by $\Lambda_{2}$. Then for any point $q\in N$ and $0<\rho<\frac{1}{2}R_{0}$, where $R_{0}$ is a positive constant, we have
   $$e^{\Lambda\rho}\rho^{-k}\mu_{V}(E_{\rho}(q))$$
   is non-decreasing in $\rho$, where $\Lambda=\Lambda(k,\Lambda_{1},\Lambda_{2},R_{0})$.
\end{theo}
\begin{proof}[Proof of Lemma \ref{lemma-21}]
First, by Nash's embedding theorem, $M$ can always be embedded in a higher Euclidean space $\bR^{n+l}$. Now, let $\Sigma$ be a singular isoperimetric hypersurface, $E_{\rho}(q)$ be the Euclidean $\rho$-ball around $q\in\Sigma$ in $\bR^{n+l}$ and $\diam(\Sigma)$ be the Euclidean diameter of the embedded $M$. By Theorem \ref{free-vol}, we have
$$\rho^{-(n-1)}\Area(E_{\rho}(q)\cap\Sigma)\leq e^{2\Lambda\cdot\diam(M)}\diam(M)^{-(n-1)}\Area(\Sigma).$$
On the other hand, since $M$ is a embedded submanifold of $\bR^{n+l}$, we compare the distance on $M$ with the Euclidean distance, then $B_{\rho}(q)\subset E_{\rho}(q)$, with $B_{\rho}(q)$ the ball of radius $\rho$ in $M$. This gives
\begin{align*}
    &\rho^{-(n-1)}\Area(B_{\rho}(q)\cap\Sigma)\\
    &\leq\rho^{-(n-1)}\Area(E_{\rho}(q)\cap\Sigma)\\
    &\leq e^{2\Lambda\cdot\diam(M)}\diam(M)^{-(n-1)}\Area(\Sigma).
\end{align*}
Thus, we can conclude that $$\cH^{n-1}(B_{\rho}(q)\cap\Sigma)\leq C\rho^{n-1},$$
for some positive constant $C$ depending on $M$, $\Sigma$, and an embedding of $M$ into Euclidean space.
\end{proof}
We next to choose a cut-off function such that it vanishes at the singular set and equals to 1 outside a small neighborhood of the singular set. Multipling this cut-off function by the outward normal vector, we can construct a geometric flow that fixes the singular set on the isoperimetric hypersurface.
Finally, we can carry out the proofs of Lemma \ref{lemm4} for $n\geq 8$.
\begin{proof}[Proof of Lemma \ref{lemm4} $(n\geq 8)$]
Let $E$ be a bounded minimizer such that $V(E)=\upsilon_{0}$ and $A(E)=I(\upsilon_{0})$. Let $K$ be a compact set with $E\subset K$. By the classical Geometric Measure theory(see free boundary case), the regular part of $\Gamma$ and $\partial \Gamma$ can be denoted by $\Gamma^{reg}$ and $\partial^{reg}\Gamma$, respectively. The singular part of the interior and boundary of $\Gamma$ can be denoted by $\Gamma^{sing}$ and $\partial^{sing}\Gamma$, respectively. According to Theorem \ref{regu-1}, we know the singular sets have Hausdorff dimension at most $n-8$.
For each $\delta<\frac{1}{4}$, we can find a finite collection of balls $B(x_{i},r_{i})$ with $x_{i}\in \Gamma^{sing}$ or $x_{i}\in \partial^{sing}\Gamma$, where $r_{i}<\delta$, such that $\sum r_{i}^{n-7}\leq 1$. For each $i$, we find a smooth function $\zeta_{i}$ such that
$$\zeta_{i}|_{B(x_{i},2r_{i})}=0,\enspace \zeta_{i}|_{M\setminus B(x_{i},3r_{i})}=1,\enspace |\nabla_{M}\zeta_{i}|\leq 2r_{i}^{-1}.$$
We claim that for each $x\in K$ and $r<1$ we have
\begin{align}\label{apen-claim}
\int_{\Gamma\cap B(x,r)}w^{\theta}\leq Cr^{n-1},
\end{align}
where $C$ depends only on $K$ and $w$. The constant $C$, might change from line to line from now on. To see this, for each $x\in M$ and $r>0$ there is a radius $s\in [0,r]$ such that $\int_{B(x,r)\setminus B(x,s)}w^{\alpha}=\int_{B(x,r)\cap E}w^{\alpha}.$ This implies that the set $$E^{'}=(E\cup B(x,r))\setminus B(x,s)$$
has $V(E^{'})=V(E)$. On the other hand, we have
$$A(E^{'})\leq\int_{\Gamma\setminus\overline{B(x,r)}}w^{\theta}+\int_{\partial^{*}B(x,r)}w^{\theta}+\int_{\partial^{*}B(x,s)}w^{\theta}\leq\int_{\Gamma\setminus\overline{B(x,r)}}w^{\theta}+Cr^{n-1},$$
and
$$A(E^{\prime})\geq A(E)\geq\int_{\Gamma\setminus\overline{B(x,r)}}w^{\theta}+\int_{\Gamma\cap B(x,r)}w^{\theta}.$$
This proves (\ref{apen-claim}). By regularizing $\bar{\zeta}:=\min\limits_{i}\{\zeta_{i}\}$, we can find a function $\zeta\in C^{\infty}(M)$ such that $$\zeta=0 \enspace on \enspace \cup B(x_{i},r_{i}),\enspace\enspace\enspace\enspace\zeta=1 \enspace on \enspace M\setminus \cup B(x_{i},4r_{i}),$$
and $|\nabla_{M}\zeta|\leq 2|\nabla_{M}\bar{\zeta}|$. Combined with (\ref{apen-claim}), and $|\nabla_{M}\zeta_{i}|\leq Cr_{i}^{-1}$, we obtain
\begin{align*}
    \int_{\Gamma^{reg}}|\nabla_{\Gamma^{reg}}\zeta|^2&\leq\int_{\Gamma^{reg}}|\nabla_{M}\zeta|^2\leq 2\sum\limits_{i}\int_{\Gamma^{reg}\cap B(x_{i},4r_{i})}|\nabla_{M}\zeta_{i}|^2\\
    &\leq C\sum\limits_{i}r_{i}^{n-1}\cdot r_{i}^{-2}\leq C\delta^{4}.
\end{align*}
For $\varphi\in C^{\infty}(M)$, let's consider a smooth family of sets $\{E_{t}\}_{t\in(-\epsilon,\epsilon)}$, such that $E_{0}=E$, the variational vector field $X_{t}$ along $\partial E_{t}$ at $t=0$ is $\varphi\zeta\nu_{\Gamma}$(where $\nu_{\Gamma}$ denotes the outer unit normal of $\Gamma$), and $\nabla_{X_{t}}X_{t}=(\varphi\zeta(\varphi\zeta)_{\nu_{\Gamma}})\nu_{\Gamma}$ at $t=0$. This family is well-defined since $\zeta$ is supported inside $\Gamma^{reg}\cup\partial^{reg}\Gamma$. The variations of the area and the volume remain unchanged as in Lemma \ref{1stvar}(with each $\varphi$ replaced with $\varphi\zeta$), let $\varphi=w^{-\theta}$, for simplicity, we write $\bangle{\nabla_{\Gamma}w,\nu_{\Gamma}}=w_{\nu_{\Gamma}}$ for example, then:

\begin{align*}
    \frac{d^2A}{dt^2}(0)&=\int_{\Gamma^{reg}}\big(-\Delta_{\Gamma^{reg}}(w^{-\theta}\zeta)-\Ric_{M}(\nu_{\Gamma},\nu_{\Gamma})w^{-\theta}\zeta-|\secondfund_{\Gamma}|^2w\zeta\big)\zeta\\
   &\quad+\big(-\theta w^{-\theta}\zeta w^{-2}w^{-2}_{\nu_{\Gamma}}+\theta w^{-1-\theta}\zeta(\Delta_{M}w-\Delta_{\Gamma}w-H_{\Gamma}w_{\nu_{\Gamma}})-\theta w^{-1}\bangle{\nabla_{\Gamma^{reg}}w,\nabla_{\Gamma^{reg}}(w^{-\theta}\zeta)}\big)\zeta\\
   &\quad+\big(\theta w^{\alpha-1}w_{\nu_{\Gamma}}w^{-2\theta}\zeta^2+w^{\alpha}w^{-\theta}\zeta(w^{-\theta}\zeta)_{\nu_{\Gamma}}+H_{\Gamma}w^{\alpha}w^{-2\theta}\zeta^2\big)w^{\theta-\alpha}\big(H_{\Gamma}+\theta u^{-1}w_{\nu_{\Gamma}}\big)\\
   &\quad+\int_{\partial^{reg}\Gamma}\zeta\bangle{\nabla_{\Gamma^{reg}}(w^{-\theta}\zeta),\nu_{\partial\Gamma}}-A_{\partial M}(\nu_{\Gamma},\nu_{\Gamma})w^{-\theta}\zeta^2.\\
\end{align*}

Next, we give some computations details to show that the boundary integral does not have extra terms. On the one hand, by integration by parts, we have
\begin{align*}
\int_{\Gamma^{reg}}-\Delta_{\Gamma^{reg}}(w^{-\theta}\zeta)\zeta&=- (\int_{\partial^{reg}\Gamma}\bangle{\nabla_{\Gamma^{reg}}(w^{-\theta}\zeta),\nu_{\partial\Gamma}}\zeta-\int_{\Gamma^{reg}}\bangle{\nabla_{\Gamma^{reg}}(w^{-\theta}\zeta),\nabla_{\Gamma^{reg}}\zeta})\\
&=-\int_{\partial^{reg}\Gamma}\bangle{\nabla_{\Gamma^{reg}}(w^{-\theta}\zeta),\nu_{\partial\Gamma}}\zeta+\int_{\Gamma^{reg}}\bangle{\nabla_{\Gamma^{reg}}w^{-\theta},\nabla_{\Gamma^{reg}}\zeta}\zeta\\
&\quad+\int_{\Gamma^{reg}}w^{-\theta}|\nabla_{\Gamma^{reg}}\zeta|^2.
\end{align*}
On the other hand,
\begin{align*}
   & \int_{\Gamma^{reg}}-\theta w^{-1-\theta}\zeta^2\Delta_{\Gamma^{reg}}w-\theta w^{-1}\bangle{\nabla_{\Gamma^{reg}w}w,\nabla_{\Gamma^{reg}}(w^{-\theta}\zeta)}\zeta\\
    &=-\theta\int_{\Gamma^{reg}}w^{-1-\theta}\zeta^2\Delta_{\Gamma^{reg}}w+w^{-1}\bangle{\nabla_{\Gamma^{reg}}w,\nabla_{\Gamma^{reg}}(w^{-\theta}\zeta)}\zeta\\
    &=-\theta\int_{\partial^{reg}\Gamma}\frac{\partial w}{\partial\nu_{\partial\Gamma}}w^{-1-\theta}\zeta^2+\theta\int_{\Gamma^{reg}}\bangle{\nabla_{\Gamma^{reg}}w,\nabla_{\Gamma^{reg}}w^{-1}}w^{-\theta}\zeta^2\\
&\quad+\theta\int_{\Gamma^{reg}}\bangle{\nabla_{\Gamma^{reg}}w,\nabla_{\Gamma^{reg}}\zeta}w^{-1-\theta}\zeta\\
&=\theta\int_{\Gamma^{reg}}\bangle{\nabla_{\Gamma^{reg}}w,\nabla_{\Gamma^{reg}}w^{-1}}w^{-\theta}\zeta^2+\theta\int_{\Gamma^{reg}}\bangle{\nabla_{\Gamma^{reg}}w,\nabla_{\Gamma^{reg}}\zeta}w^{-1-\theta}\zeta\\
&\leq\theta\int_{\Gamma^{reg}}\bangle{\nabla_{\Gamma^{reg}}w,\nabla_{\Gamma^{reg}}\zeta}w^{-1-\theta}\zeta\\
&=-\int_{\Gamma^{reg}}\bangle{\nabla_{\Gamma^{reg}}w^{-\theta},\nabla_{\Gamma^{reg}}\zeta}\zeta.
\end{align*}
Following the argument of Lemma \ref{lemm4}, we have $w^{\theta-\alpha}(H_{\Gamma}+\theta w^{-1}w_{\nu_{\Gamma}})=A^{\prime}(\upsilon_{0})$ on $\Gamma^{reg}$. Conduct the same computations as in Theorem \ref{diamter theo}, since $M$ has weakly convex boundary, we finally obtain
\begin{align*}
    \frac{d^2A}{dt^2}(0)&\leq\int_{\Gamma^{reg}}w^{-\theta}|\nabla_{\Gamma^{reg}}\zeta|^2+w^{\alpha-2\theta}\zeta\zeta_{\nu_{\Gamma}}A^{\prime}(\upsilon_{0})\\
    &-(n-1)\lambda w^{-\theta}\zeta^2+w^{-\theta}\zeta^2[\frac{n-2}{n-1}X^2+\frac{2-n}{n-1}2\theta XY+(\frac{n-2}{n-1}\theta^2-\theta)Y^2].
\end{align*}
The other parts are the same as the closed case, we refer the readers to the paper \cite{Antonelli-Xu}.
\end{proof}
\begin{rema}
    For Theorem \ref{diamter theo}, its proof is a similar process with slight modifications. Combining the weak convexity of $M$, we can also prove that Theorem \ref{diamter theo} holds for $n\geq 8$. Other processes are similar to those of \cite[Appendix A]{Antonelli-Xu}.
\end{rema}
\section*{Appendix A. Proof of the second variation Formula for Free boundary $\mu$-bubbles}
\begin{proof}[Proof of Lemma \ref{first-area}]
    The first variation formula of $\cA$ is a direct computation, and the boundary integral vanishes because of the free boundary condition. In this part, we aim to prove the second variation formula.


     Let $dA$ be the area element of $\Sigma$ and $ds$ be the area element of $\partial \Sigma$. For other cases, we always omit the notations elsewhere. 
    First, the first variation is a direct computation, and we find the minimizer satisfies that the boundary $\Sigma$ intersects with $\partial N$ orthogonally.
   Differentiate it twice, we have
   \begin{align}\label{sec-vara}
   \notag    \frac{d^{2}}{dt^{2}}\big|_{t=0}\cA(\Omega_{t})&=\int_{\Sigma}\left(\partial_{t}(H_{\Sigma}+\theta\partial_{t}(w^{\theta-1})\bangle{\nabla^{N}w,\nu_{\Sigma}}+\theta w^{-1}\partial_{t}(\bangle{\nabla^{N}w,\nu_{\Sigma}}-\partial_{t}(hw^{\alpha-\theta})\right)w^{\theta}\phi dA\\
       &\quad+\int_{\Sigma}(H_{\Sigma}+\theta w^{-1}\bangle{\nabla^{N}w,\nu_{\Sigma}}-hw^{\alpha-\theta})w^{\theta}\phi\partial_{t}(dA)\notag\\
&\quad+\int_{\partial\Sigma}\partial_{t}(w^{\theta})\bangle{\phi\nu_{\Sigma},\nu_{\partial\Sigma}}ds+\int_{\partial\Sigma}\phi w^{\theta}\bangle{\partial_{t}\nu_{\Sigma_{t}},\nu_{\partial\Sigma}}ds+\int_{\partial\Sigma}\phi w^{\theta}\bangle{\nu_{\Sigma_{t}},\partial_{t}\nu_{\partial\Sigma}}ds\notag\\
&\quad+\int_{\partial\Sigma}\phi w^{\theta}\bangle{\nu_{\Sigma_{t}},\nu_{\partial\Sigma}}\partial_{t}(ds)\notag\\
&=\int_{\Sigma}\left(\partial_{t}(H_{\Sigma}+\theta\partial_{t}(w^{\theta-1})\bangle{\nabla^{N}w,\nu_{\Sigma}}+\theta w^{-1}\partial_{t}(\bangle{\nabla^{N}w,\nu_{\Sigma}}-\partial_{t}(hw^{\alpha-\theta})\right)w^{\theta}\phi dA\notag\\
&\quad
+\int_{\partial\Sigma}\phi w^{\theta}\bangle{\partial_{t}\nu_{\Sigma_{t}},\nu_{\partial\Sigma}}ds+\int_{\partial\Sigma}\phi w^{\theta}\bangle{\nu_{\Sigma_{t}},\partial_{t}\nu_{\partial\Sigma}}ds,\notag\\
\notag\end{align}
 where the second term and the last term vanish because of the critical condition. The interior integrals are the same as the \cite[Section 2]{Antonelli-Xu}, we only need to deal with the second and third terms. 
For convenience, we use “prime” to denote the covariant derivative about $t$.

This part follows essentially the same argument as in \cite[Appendix]{Souam1997}. In what follows, we present a more general proof that does not assume the normal variation conditions. Suppose that $\Phi(t)$ is an variation of $\Sigma$ and $X$ is the corresponding variational vector field at $t=0$. We denote by $\phi=\bangle{X,\nu_{\Sigma}}$. According the first variation information, $X$ is tangent to $\partial N$ along $\partial N$. To state it clearly, we adopt the following conventions. 
\begin{itemize}
\item We denote $\nabla^{N}$ and $\nabla$ be the Levi-Civita connection of $N$ and $\Sigma$ respectively.
    \item We denote by $X_0$ and $X_1$ the tangent components of the vector field $X$ to $\Sigma$ and $\partial\Sigma$, respectively. 
    \item Let $A_{\partial N}$ be the second fundamental form of $\partial N$ in $N$ with respect to $\eta$;
    \item Let $S_{0}$ be the shape operator of $\Sigma$ in $N$ with respect to $\nu_{\Sigma}$;
    \item Let $S_{1}$ be the shape operator of $\partial\Sigma$ in $\Sigma$ with respect to $\nu_{\partial\Sigma}$;
    \item Since the minimizer $\Sigma$ has free boundary along $\partial N$, it follows that $\bar{\upsilon}=\nu_{\Sigma}$, $\eta=\nu_{\partial\Sigma}$. Where $\bar{\upsilon}$ is the outward unit normal of $\partial\Sigma$ in $\partial N$, and $\eta$ is the outward unit normal of $\partial N$ in $N$.
\end{itemize}
Additionally, we define $\sigma(Y,Z)=\bangle{\nabla^{N}_{Y}\nu_{\Sigma},Z}$, $S_{0}(\eta)=(\nabla^{N}_{\eta}\nu_{\Sigma})^{T_{0}}$, $S_{1}(X_{1})=({\nabla}_{X_{1}}\nu_{\partial\Sigma})^{T_{1}}$, and $S_{2}(X_{1})=(\nabla^{N}_{X_{1}}\bar{\upsilon})^{T_{2}}$, where $T_{0}$ denotes the tangential projection to $T_{p}\Sigma$, other cases are defined similarly.

We choose an orthonormal basis $\{e_{j}\}_{j=1}^{n-1}$ of $T_{p}\Sigma$, $e_{j}(t)=d\Phi_{t}(e_{j})$, we just write $e_{j}$ for simplicity. We find
\begin{align*}
\bangle{\nu_{\Sigma},\nabla^{N}_{e_{j}}X}&=\langle{\nu_{\Sigma},\nabla^{N}_{e_{j}}X^{\perp}}\rangle+\langle{\upsilon,\nabla^{N}_{e_{j}}X^{T}}\rangle\\
&=\langle{\nu_{\Sigma},\nabla^{N}_{e_{j}}\bangle{X,\nu_{\Sigma}}\nu_{\Sigma}}\rangle+\langle{\nu_{\Sigma},\nabla^{N}_{e_{j}}X_{0}}\rangle\\
&=d\phi(e_{j})+\langle{\upsilon,\nabla^{N}_{e_{j}}X_{0}}\rangle.
\end{align*}
Since $\bangle{\nu_{\Sigma}({t}),e_{j}(t)}=0$ and $[e_{j}(t),X_{t}]=0$, then 
\begin{align*}
\nu_{\Sigma}^{'}&=\sum\limits_{j=1}^{n-1}\langle{\nu_{\Sigma}^{'},e_{j}}\rangle e_{j}\\
&=-\sum\limits_{j=1}^{n-1}\langle{\nu_{\Sigma},e_{j}^{'}}\rangle e_{j}\\
&=-\sum\limits_{j=1}^{n-1}\langle{\nu_{\Sigma},\nabla^{N}_{e_{j}}X}\rangle e_{j}\\
&=-\sum\limits_{j=1}^{n-1}d\phi(e_{j})e_{j}-\sum\limits_{j=1}^{n-1}\langle{\upsilon,\nabla^{N}_{e_{j}}X_{0}}\rangle e_{j}\\
&=-\sum\limits_{j=1}^{n-1}d\phi(e_{j})e_{j}+\sum\limits_{j=1}^{n-1}\langle{\nabla^{N}_{e_{j}}\upsilon,X_{0}}\rangle e_{j}\\&=-\sum\limits_{j=1}^{n}d\phi(e_{j})e_{j}+\sum\limits_{j=1}^{n}\langle{\nabla^{N}_{X_{0}}\upsilon,e_{j}}\rangle e_{j}\\
&=-\nabla \phi+S_{0}(X_{0}).
\end{align*}
Hence, we obtain
\begin{align*}
\nu_{\Sigma}^{'}=-\nabla \phi+S_{0}(X_{0}).
\end{align*}
Next, we aim to derive the expression of $\nu_{\partial\Sigma}^{'}$.
Since $\nu_{\partial\Sigma}^{'}\perp\nu_{\partial\Sigma}$, we can decompose $\nu_{\partial\Sigma}$ in the direction of $T_{p}\partial\Sigma$ and the unit normal $\nu_{\Sigma}$ of $\Sigma$ along $\partial\Sigma$. We let $\{\nu_{k}\}_{k=1}^{n-2}$ be an orthonormal basis of $T_{p}\partial\Sigma$ for any fixed $p\in\partial\Sigma$, where $\nu_{k}(t)=d\Phi_{t}(\nu_{k})$, $k=1,\cdots,n-2$. 
On the one hand, we note that 
\[
    \langle{\nu_{\partial\Sigma}^{'},\nu_{\Sigma}}\rangle=-\langle{\nu_{\partial\Sigma},\nu_{\Sigma}^{'}}\rangle=-\langle{\nu_{\partial\Sigma},-\nabla \phi+S_{0}(X_{0})}\rangle=\frac{\partial \phi}{\partial\nu_{\partial\Sigma}}-\sigma(X_{0},\nu_{\partial\Sigma}).
\]
On the other hand, $\nu_{\partial\Sigma}=\eta$, and we write $X$ along $\partial\Sigma$ as
\[
    X=\langle{X,\nu_{\Sigma}}\rangle\nu_{\Sigma}+X_{1}=\phi\nu_{\Sigma}+X_{1},
\]
so 
\begin{align*}
\langle{\nu_{\partial\Sigma},\nu_{k}}\rangle\nu_{k}&=-\langle{\nu_{\partial\Sigma},\nu_{k}^{'}}\rangle\nu_{k}\\
&=-\langle{\nu_{\partial\Sigma},\nabla^{N}_{\nu_{k}}X}\rangle\nu_{k}\\
&=-\phi\langle{\nu_{\partial\Sigma},\nabla^{N}_{\nu_{k}}\nu_{\Sigma}}\rangle\nu_{k}-\langle{\nu_{\partial\Sigma},\nabla^{N}_{\nu_{k}}X_{1}}\rangle\nu_{k}\\
&=-\phi\langle{\nabla^{N}_{\nu_{\partial\Sigma}}\nu_{\Sigma},\nu_{k}}\rangle\nu_{k}+\langle{\nabla^{N}_{\nu_{k}}\nu_{\partial\Sigma},X_{1}}\rangle\nu_{k}.
\end{align*}
 Hence $$\sum\limits_{k=1}^{n-2}\langle{\nu_{\partial\Sigma},\nu_{k}}\rangle\nu_{k}=-\phi S_{0}(\nu_{\partial\Sigma})+\phi\sigma(\nu_{\partial\Sigma},\nu_{\partial\Sigma})\nu_{\partial\Sigma}+S_{1}(X_{1}).$$
Therefore, we obtain the expression of $\nu_{\partial\Sigma}$:
\begin{align}\label{deriva-2}
\nu_{\partial\Sigma}^{'}&=\langle{\nu_{\partial\Sigma}^{'},\nu_{\Sigma}}\rangle+\sum\limits_{k=1}^{n-2}\langle{\nu_{\partial\Sigma}^{'},\nu_{k}}\rangle\nu_{k}+\langle{\nu_{\partial\Sigma}^{'},\nu_{\partial\Sigma}}\rangle\notag\\
&=(\frac{\partial \phi}{\partial\nu_{\partial\Sigma}}-\sigma(X_{0},\nu_{\partial\Sigma}))\nu_{\Sigma}-\phi S_{0}(\nu_{\partial\Sigma})+\phi\sigma(\nu_{\partial\Sigma},\nu_{\partial\Sigma})\nu_{\partial\Sigma}+S_{1}(X_{1}).
\end{align}
Now, we return to compute $\langle{X,\nu_{\partial\Sigma}^{'}}\rangle.$
\begin{align*}
&\langle{X,\nu_{\partial\Sigma}^{'}}\rangle\\
&=\langle{X,(\frac{\partial \phi}{\partial\nu_{\partial\Sigma}}-\sigma(X_{0},\nu_{\partial\Sigma}))\nu_{\Sigma}}\rangle-\phi\langle{S_{0}(\nu_{\partial\Sigma}),X}\rangle+\sigma(\nu_{\partial\Sigma},\nu_{\partial\Sigma})\langle{X,\nu_{\partial\Sigma}}\rangle+\langle{S_{1}(X_{1}),X}\rangle\\
&=\phi\frac{\partial \phi}{\partial\nu_{\partial\Sigma}}-\phi\sigma(X_{0},\nu_{\partial\Sigma})-\phi\langle{S_{0}(\nu_{\partial\Sigma}),X}\rangle+\sigma(\nu_{\partial\Sigma},\nu_{\partial\Sigma})\langle{X,\nu_{\partial\Sigma}}\rangle+\langle{(\nabla^{N}_{X_{1}}\nu_{\partial\Sigma})^{T_{1}},X}\rangle\\
&=\phi\frac{\partial \phi}{\partial\nu_{\partial\Sigma}}-\phi\sigma(X_{0},\nu_{\partial\Sigma})-\phi\langle{(\nabla^{N}_{\nu_{\partial\Sigma}}\nu_{\Sigma})^{T_{0}},f\nu_{\Sigma}+X_{1}}\rangle+\phi\langle{(\nabla^{N}_{\nu_{\partial\Sigma}}\nu_{\Sigma})^{T_{0}},\nu_{\partial\Sigma}}\rangle\langle{\nu_{\partial\Sigma},\phi\nu_{\Sigma}+X_{1}}\rangle\\
&\quad+\langle{(\nabla^{N}_{X_{1}}\nu_{\partial\Sigma})^{T_{1}},f\nu_{\Sigma}+X_{1}}\rangle\\
&=\phi\frac{\partial \phi}{\partial\nu_{\partial\Sigma}}-\phi\sigma(X_{0},\nu_{\partial\Sigma})-\phi\langle{\nabla^{N}_{\nu_{\partial\Sigma}}\nu_{\Sigma},X_{1}}\rangle+\langle{\nabla^{N}_{X_{1}}\nu_{\partial\Sigma},X_{1}}\rangle\\
&=\phi\frac{\partial \phi}{\partial\nu_{\partial\Sigma}}-\phi\sigma(X_{0},\nu_{\partial\Sigma})-\phi\langle{\nabla^{N}_{\nu_{\partial\Sigma}}\nu_{\Sigma},X_{0}}\rangle+\langle{\nabla^{N}_{X_{1}}\nu_{\partial\Sigma},X_{1}}\rangle\\
&=\phi\frac{\partial \phi}{\partial\nu_{\partial\Sigma}}-2\phi\sigma(X_{0},\nu_{\partial\Sigma})+A_{\partial N}(X_{1},X_{1}),
\end{align*}
where we used that the fact $X_{0}=X_{1}$ along $\partial\Sigma$, so 
\[
  -2\phi\sigma(X_{0},\nu_{\partial\Sigma})=-2\phi\sigma(X_{1},\nu_{\partial\Sigma}),
\]
hence we obtain 
\[
    \langle{X,\nu_{\partial\Sigma}^{'}}\rangle=\phi\frac{\partial \phi}{\partial\nu_{\partial\Sigma}}-2\phi\sigma(X_{1},\nu_{\partial\Sigma})+A_{\partial N}(X_{1},X_{1}).
\]
In addition, $$\langle{X^{'},\nu_{\partial\Sigma}}\rangle=\langle{X^{'},\eta}\rangle=-A_{\partial N}(X,X).$$
Using the fact $X=X_{1}+\phi\nu_{\Sigma}$, we have
$$A_{\partial N}(X,X)=A_{\partial N}(X_{1},X_{1})+\phi^{2}A_{\partial N}(\nu_{\Sigma},\nu_{\Sigma})+2\phi A_{\partial N}(X_{1},\nu_{\Sigma}).$$
Since 
\begin{align*}
\sigma(X_{1},\nu_{\partial\Sigma})&=\langle{\nabla^{N}_{X_{1}}\nu_{\Sigma},\nu_{\partial\Sigma}}\rangle
=-\langle{\nabla^{N}_{X_{1}}\nu_{\partial\Sigma},\nu_{\Sigma}}\rangle
=-A_{\partial N}(X_{1},\nu_{\Sigma}).
\end{align*}

Finally, we have
\begin{align*}
&\langle{X,\nu_{\partial\Sigma}^{'}}\rangle+\langle{X^{'},\nu_{\partial\Sigma}}\rangle\\
&=\phi\frac{\partial \phi}{\partial\nu_{\partial\Sigma}}-2\phi\sigma(X_{1},\nu_{\partial\Sigma})+A_{\partial N}(X_{1},X_{1})\\
&\quad-\left(A_{\partial N}(X_{1},X_{1})+\phi^{2}A_{\partial N}(\nu_{\Sigma},\nu_{\Sigma})+2\phi A_{\partial N}(X_{1},\eta)\right)\\
&=\phi\frac{\partial \phi}{\partial \nu_{\partial\Sigma}}-\left(A_{\partial N}(\nu_{\Sigma},\nu_{\Sigma})\right)\phi^2.\\
\end{align*}
Putting them into the second variation formula, we have 
\begin{align*}
  \frac{d^2}{dt^{2}}\big|_{t=0}\cA(\Omega_{t})&=\int_{\Sigma}\big[-\Delta_{\Sigma}\phi-|\secondfund_{\Sigma}|^2\phi-\Ric_{N}(\nu_{\Sigma},\nu_{\Sigma})\phi-\theta w^{-2}\bangle{\nabla^{N}w,\nu_{\Sigma}}^2\phi\\
  &\quad+\theta w^{-1}\phi(\Delta_{N}w-\Delta_{\Sigma}w-H_{\Sigma}\bangle{\nabla^{N}w,\nu_{\Sigma}})-\theta w^{-1}\bangle{\nabla^{\Sigma}w,\nabla^{\Sigma}\phi}\\
  &\quad-\phi\bangle{\nabla^{N}h,\nu_{\Sigma}}w^{\alpha-\theta}+(\theta-\alpha)w^{\alpha-\theta-1}h\phi\bangle{\nabla^{N}w,\nu_{\Sigma}}\big]w^{\theta}\phi\\
&\quad+\int_{\partial\Sigma}w^{\theta}\phi\frac{\partial \phi}{\partial\nu_{\partial \Sigma}}-A_{\partial N}(\nu_{\Sigma},\nu_{\Sigma})\phi^2w^{\theta}.
\end{align*}
This completes the proof.
\end{proof}
\bibliographystyle{plain}

\end{document}